\documentclass[11pt]{amsart}

\usepackage{amssymb}   
\usepackage[all]{xy}   

\usepackage[centering,margin=1in]{geometry}

\usepackage{stmaryrd}
\usepackage[colorlinks=true]{hyperref}       

\theoremstyle{plain}
\newtheorem{lem}{Lemma}[section]
\newtheorem{cor}[lem]{Corollary}
\newtheorem{prop}[lem]{Proposition}
\newtheorem{thm}[lem]{Theorem}

\theoremstyle{definition}
\newtheorem{ex}[lem]{Example}
\newtheorem{rem}[lem]{Remark}
\newtheorem{dfn}[lem]{Definition}
\newtheorem{conj}[lem]{Conjecture}

\newcommand{\Ss}{\mathbf{S}}   
\newcommand{\Rr}{\mathbf{R}}  
\newcommand{\Qq}{\mathbf{Q}}   
\newcommand{\DF}{\mathbf{D}_F}   
\newcommand{\Z}{\mathbb{Z}}       
\newcommand{\unit}{\mathbf{1}}     

\newcommand{\al}{\alpha}
\newcommand{\be}{\beta}
\newcommand{\de}{\delta}
\newcommand{\la}{\lambda}
\newcommand{\La}{\Lambda}

\newcommand{\hh}{\operatorname{\mathtt{h}}} 
\newcommand{\pt}{\operatorname{pt}}   
\newcommand{\End}{\operatorname{End}}     
\newcommand{\im}{\operatorname{im}}            

\newcommand{\SW}{\Ss_W}    
\newcommand{\SWd}{\Ss_W^\star}     
\newcommand{\QW}{\Qq_W}      
\newcommand{\QWd}{\Qq_W^*}     
\newcommand{\DFd}{\DF^\star}   
\newcommand{\DFP}[1]{\mathbf{D}_{F,#1}} 
\newcommand{\DFPd}[1]{\mathbf{D}_{F,#1}^\star} 

\newcommand{\nmxiX}[1]{\hat\zeta_{#1}^{F_m, J}}  
\newcommand{\nmxiY}[1]{\zeta_{#1}^{F_m, J}}        
\newcommand{\naxi}[1]{\zeta_{#1}^{F_a, J}}            

\newcommand{\frakm}{\mathfrak{m}}  
\newcommand{\frakn}{\mathfrak{n}}
\newcommand{\mX}[2]{\hat\zeta^{#1}_{I_{#2}}}  
\newcommand{\mY}[2]{\zeta^{#1}_{I_{#2}}}       
\newcommand{\mXG}[2]{\hat\zeta^{#1}_{{#2}}}  
\newcommand{\mYG}[2]{\zeta^{#1}_{{#2}}}       

\newcommand{\calM}{\mathcal{M}}
\newcommand{\calN}{\mathcal{N}}

\newcommand{\DFm}{{D}_m}
\newcommand{\DFmd}{{D}_{m}^\star}
\newcommand{\DFmdP}[1]{{D}^\star_{m,#1}}

\newcommand{\Dh}{D_{F_t}}
\newcommand{\DFh}{\mathbf{D}_{F_t}}
\newcommand{\DFhd}{\mathbf{D}_{F_t}^\star}
\newcommand{\DFhdP}[1]{\mathbf{D}^\star_{F_t,#1}}

\begin{document}

\title[Kazhdan-Lusztig basis and Schubert classes]{Parabolic Kazhdan-Lusztig basis, Schubert classes, \\and equivariant oriented cohomology}

\author[C.~Lenart]{Cristian Lenart}
\address[Cristian Lenart]{Department of Mathematics and Statistics, State University of New York at Albany, 
Albany, NY 12222, U.S.A.}
\email{clenart@albany.edu}
\urladdr{http://www.albany.edu/\~{}lenart/}

\author[K.~Zainoulline]{Kirill Zainoulline}
\address[Kirill Zainoulline]{Department of Mathematics and Statistics, University of Ottawa, 585 King Edward Street, Ottawa, ON, K1N 6N5, Canada}
\email{kirill@uottawa.ca}
\urladdr{http://mysite.science.uottawa.ca/kzaynull/}

\author[C. Zhong]{Changlong Zhong}
\address[Changlong Zhong]{Department of Mathematics and Statistics, State University of New York at Albany, 
Albany, NY 12222, U.S.A.}
\email{czhong@albany.edu}
\urladdr{http://www.albany.edu/\~{}cz954339/}


\begin{abstract} 
We study the equivariant oriented cohomology ring $\hh_T(G/P)$ of partial flag varieties using the moment map approach. We define the right Hecke action on this cohomology ring, and then prove that the respective Bott-Samelson classes in $\hh_{T}(G/P)$  can be obtained by applying this action to the fundamental class of the identity point, hence generalizing previously known results by Brion, Knutson, Peterson, Tymoczko and others. 

We then focus on the equivariant  oriented cohomology theory corresponding to the 2-parameter Todd genus. We give a new interpretation of Deodhar's construction of the parabolic Kazhdan-Lusztig basis. Based on it, we define the parabolic Kazhdan-Lusztig (KL) Schubert classes independently of a reduced word. We make a positivity conjecture, and a conjecture about the relationship of such classes with smoothness of Schubert varieties. We then  prove several special cases.
\end{abstract}

\maketitle

\tableofcontents

\section{Introduction}

Let $G$ be a split semisimple linear algebraic group over a field $k$ of characteristic 0, with $G\supset P\supset B\supset T$, where $P$ is a parabolic subgroup, $B$ a Borel subgroup, and $T$ a maximal torus. The partial flag variety $G/P$  has many remarkable geometric properties; for instance, it has a $B$-equivariant cellular filtration given by $B$-orbits $\mathcal{O}_w$.
Basic examples of such varieties are projective spaces, Grassmannians, and  smooth quadric hypersurfaces.
In the present paper we study the cohomology ring $\hh_T(G/P)$, where $\hh_T$ is the $T$-equivariant oriented cohomology theory in the sense of Levine-Morel \cite{LM}.
Observe that the Chow groups $CH$ (or singular cohomology) and Grothendieck's $K_0$ are standard examples of such theories. The corresponding  equivariant versions have been introduced and extensively studied by Brion \cite{Br97}, Totaro \cite{To}, Edidin-Graham \cite{EG}, and others. A universal example of such a theory is given by the algebraic cobordism~$\Omega$ of Levine-Morel \cite{LM}. Its $T$-equivariant version $\Omega_T$ has been extensively studied recently (see e.g. \cite{HM13}).
An interesting family of examples comes from formal group laws of (singular) elliptic curves, which give, for instance, stalk versions of the elliptic cohomology of Ginzburg-Kapranov-Vasserot (see e.g. \cite{ZZ}).

The fact that $G/P$ has an equivariant cellular filtration immediately implies that $\hh_T(G/P)$ is a free module over the coefficient ring $\hh_T(\pt)$, with a basis
given by the classes of desingularizations of the orbit closures $X(w)=\overline{\mathcal{O}_w}$; the varieties $X(w)$ are the Schubert varieties.
Hence, the problem of describing the ring $\hh_T(G/P)$ (one of the major problems in generalized Schubert calculus) splits into two parts: (i) constructing a desingularization $\hat X(w)$ for each $X(w)$, and (ii) expressing
the products $[\hat X(w)]\cdot [\hat X(v)]$ as linear combinations of the classes $[\hat X(u)]$.
A fundamental result of Bressler-Evens \cite{BE90} says that the classes $[\hat X(w)]$ do not depend on the choice of a desingularization 
if and only if the theory $\hh$ is obtained by specialization
 from connective $K$-theory 
(a universal birational theory). 
In the latter case (e.g. for Chow groups, singular cohomology, and usual $K$-theory), the situation drastically simplifies, as 
we can replace $[\hat X(w)]$ with the class of the Schubert variety $[X(w)]$ itself,  hence avoiding  the problem~(i). 
In all other cases (e.g. for cobordism or for an elliptic cohomology), this
leads to a natural question of constructing canonical classes in $\hh_T(G/P)$ 
which serve as replacements for $[X(w)]$.

In the present paper we address this question by introducing and applying a new tool which we call a {\em right Hecke action} on $\hh_T(G/P)$.

Historically, the classical left Hecke action (denoted `$\bullet$') was used by Demazure \cite{De73, De74}, by Bernstein-Gelfand-Gelfand \cite{BGG} to construct Schubert classes for ordinary Chow groups, by Arabia \cite{Ar86, Ar89} for $T$-equivariant Chow groups, and by Kostant-Kumar \cite{KK86, KK90} for $T$-equivariant Chow groups and $T$-equivariant $K$-theory. 
Namely, for Chow groups one sets $[X(w)]:=Y_J\bullet (Y_{w^{-1}}\bullet [\pt_J])$, where $J$ is the subset of simple roots corresponding to $P$, $[\pt_J]$ is the fundamental class of the identity point in $CH_T(G/P)$, $Y_{w^{-1}}$ is the push-pull element  in the nil-Hecke ring corresponding to $\mathcal{O}_{w}$, and $Y_J$ is the averaging element. However, the main computational disadvantage of the $\bullet$-action is that it does not generate the Schubert classes (from the class of a point) directly (without applying the averaging element), because $Y_v\bullet [\pt_J]$ does not produce all Schubert classes.
In order to fix this problem (in the context of Chow groups), a new Hecke action was introduced by Brion \cite{Br97}, Peterson \cite{Pe97} and Knutson \cite{Ty08}. It was later studied by Tymoczko in \cite{Ty08, Ty09} (sometimes it is called the Tymoczko action). Geometrically speaking, it coincides with a natural left action of the Weyl group $W$ on the equivariant Chow groups $CH_T(G/P)=CH_G(G/T\times G/P)$ induced by the right $G$-equivariant action of $W$ on $G/T\times G/P$ via  $(gT,x)\cdot \sigma T=(gT\sigma,x)$.  

The next step was done in \cite{CZZ1} and \cite{CZZ2}, where the classical $\bullet$-action was generalized to an arbitrary oriented theory $\hh_T(G/P)$ to construct
classes of the Bott-Samelson desingularizations $\zeta_{I_w}^J$  of the Schubert varieties $X(w)$. 
Namely, one sets $\zeta_{I_w}^J:=Y_J\bullet (Y_{I_{w}^{-1}}\bullet [\pt_J])$, where $Y_{I_{w}^{-1}}$ is the push-pull element in the respective formal affine Demazure algebra, which depends on the choice of a reduced word $I_w$ for $w$. 

In the present paper (see~\S\ref{sec:twac}), we generalize the Brion-Peterson-Knutson-Tymoczko action to an arbitrary oriented  theory $\hh$ (we denote it by `$\odot$' and call it the right Hecke action), hence, completing the picture. Our arguments are based on the following geometric observation:
Consider the convolution product on $CH_T(G/B)\cong CH_G(G/B\times G/B)$; 
taken from the left it gives the usual $\bullet$-action, 
and taken from the right it gives the Tymoczko action. It follows immediately that these two actions commute. Let $W_P\subset W$ be the subgroup corresponding to $P\supset B$. By identifying $CH_T(G/P)$ as the $W_P$-invariant subring via the $\bullet$-action, the Tymoczko action  then restricts to an action on $CH_T(G/P)$.  From this interpretation, the latter naturally extends to any equivariant oriented cohomology theory via $\hh_T(G/B)\cong \hh_{G}(G/B\times G/B)$.

 As in the case of Chow groups, the $\odot$-action satisfies the following key property (see Theorem~\ref{thm:main}).

\begin{quote}
All parabolic Bott-Samelson classes $\zeta_{I_w}^J$ in $\hh_T(G/P)$ can be obtained as $Y_{I_w}\odot [\pt_J]$.
\end{quote}
Using basic properties of the `$\odot$'-action (proven in \S\ref{subset:two}-\ref{heckeactions}), we provide several explicit formulas (see \S\ref{subsec:BS} and \S\ref{subsec:recform}), which give previously know results for Chow groups (Example~\ref{ex:Chowformula}), and new results for $K$-theory (Example~\ref{ex:formulaK}).

The $\odot$-action turns out to be a natural tool to generalize some of Deodhar's fundamental work in  \cite{D87} to oriented cohomology theories; 
this is discussed in the second part of the paper.

For instance, using only the basic properties of the `$\odot$'-action, one can generalize the Deodhar acyclic complex for the Iwahori-Hecke algebra \cite[\S5]{D87} (see \S\ref{subsec:deod}).

In section~\ref{sec:Deodmult} we give a new interpretation of Deodhar's modules over Iwahori-Hecke algebras in terms of the equivariant $K$-theory of flag varieties. All the results stated here are known, but we provide proofs in the new setup (using the '$\odot$'-action), which give new insight and are sometimes simpler.  Moreover, the basis  in Deodhar's modules discussed in \S\ref{sec:Deodmult} is closely related to the  {\em stable basis} in the $K$-theory of the Springer resolution defined by Maulik and Okounkov~\cite{MO}; an investigation of the latter is contained in \cite{SZZ}. 

In the last section we give new interpretation of  the Kazhdan-Lusztig (KL) theory in the context of oriented cohomology corresponding to a $2$-parameter Todd genus 
(e.g., multiplicative and generic hyperbolic). Most of the results here are new.
As a major application, we define the parabolic KL-Schubert classes (see \S\ref{subsec:parKL}) which do not depend on the choice of reduced decompositions. We discuss their functorial properties (see \S\ref{subsec:functor}), and  recover classical relationship between Kazhdan-Lusztig basis and Deodhar's parabolic version in \ref{sec:paraBorel}.  We state several conjectures (smoothness \S\ref{subsec:smooth} and positivity \S\ref{subsec:positive}), and prove some special cases. Our main motivation was Conjecture~\ref{smoothconj}, which says  that if the Schubert variety $X(w)$ in $G/P$ is smooth, then the  Schubert class $[X(w)]$ whose restriction formula is given in {\rm \eqref{smooth}} coincides with the normalized parabolic KL-Schubert class indexed by $w$. In~\S\ref{subsec:projective}
we compute the parabolic KL-Schubert classes and prove this conjecture for projective spaces.

\medskip

\noindent{\bf Acknowledgments.}
We thank William Graham and Samuel Evens for valuable input.

C.L. was partially supported by the NSF grant DMS--1362627. K.Z. was partially supported by the NSERC Discovery Grant RGPIN-2015-04469; K.Z. is also grateful to I.H.E.S. and Universit\'e Paris 13 for hospitality and support. C.Z. was partially supported by the Simons Foundation and
by the Mathematisches Forschungsinstitut Oberwolfach (MFO) under the program of Simons Visiting Professorship, and also supported by  Marc Levine during a research stay at the University of Duisburg-Essen.

\section{Preliminaries and notation}

\subsection{The root datum} 
We consider a semi-simple root datum $(\La,\Sigma,\Sigma\hookrightarrow\La^\vee)$ following \cite[Exp.~XXI]{SGA} (cf. \cite[\S2]{CZZ}). Here $\La$ is a free abelian group of finite rank, $\La^\vee$ is its dual, $\Sigma$ is a finite non-empty subset of $\La$, called the set of roots and $\Sigma\hookrightarrow \La ^\vee$, $\al\mapsto \al^\vee$ is an embedding satisfying standard axioms. The subgroup $\La_r$ generated by $\Sigma$ is called the root lattice, and the subgroup \[\La_w=\{\la\in \La\otimes_\Z\mathbb{Q}\mid \langle \la, \al^\vee\rangle\in \Z \text{ for all }\al\in \Sigma\}\] is called the weight lattice. By definition, we have $\La_r\subseteq \La\subseteq \La_w$. 

The root lattice $\La_r$ has a $\Z$-basis $\Pi=\{\al_1,\ldots,\al_n\}$, called the set of simple roots. Here $n$ is the rank of the root datum. Any root $\al\in \Sigma$ can be written as a linear combination of simple roots with coefficients either all positive or all negative. So we have a decomposition $\Sigma=\Sigma^+\amalg \Sigma^-$  into positive and negative roots. 

For any $\al\in \Sigma$, the $\Z$-linear operator \[s_\al\colon \La\to \La,\quad \la\mapsto \la-\langle \la, \al^\vee\rangle \al\] is called a reflection. The Weyl group $W$ of the root datum is generated by $s_\al$, $\al\in \Sigma$; it is also generated by simple reflections $s_i$, $i\in [n]$, where $s_i:=s_{\al_i}$ and $[n]=1,\ldots,n$.  Sometimes we abuse notation and write $\Pi$ instead of $[n]$. The length of $w\in W$ is denoted by $\ell(w)$, and if $w=s_{i_1}\cdots s_{i_l}$ is a reduced decomposition into a product of simple reflections of length $l$, we say that $I_w:=(i_1,\ldots,i_l)$ is a reduced sequence for $w$. The longest element of $W$ with respect to the length is denoted by $w_\circ$.

Let $J\subseteq [n]$ and let $W_J$ be the subgroup of $W$ generated by $s_i$,  $i\in J$. Then $W_\emptyset=\{e\}$ is the trivial group and $W_{[n]}=W$.  We denote by $W^J=\{v\in W\mid\ell(vs_i)> \ell(v) \text{ for all }i\in J\}$ the subset of minimal left coset representatives of $W/W_J$. We say that a set of reduced sequences $\{I_w\}_{w\in W}$ is $J$-compatible if $w=uv$, $u\in W^J$, $v\in W_J$ implies that $I_w$ is the concatenation of $I_u$ and $I_v$. We set \[\Sigma_J=\{\al\in \Sigma\mid s_\al\in W_J\},\quad \Sigma^+_J=\Sigma_J\cap \Sigma^+\;\text{ and }\; \Sigma^-_J=\Sigma_J\cap \Sigma^-.\] 

We will extensively use the following classical fact, sometimes implicitly.
\begin{lem}\cite[Lemma~2.1]{D87}\label{lem:parw} For any $i\in [n]$ and $v\in W^J$, exactly one of the following three possibilities occurs:
\begin{enumerate}
\item[{\rm (i)}] $\ell(s_iv)<\ell(v)$, in which case $s_iv\in W^J$ as well;
\item[{\rm (ii)}] $\ell(s_iv)>\ell(v)$ and $s_iv\in W^J$;
\item[{\rm (iii)}] $\ell(s_iv)>\ell(v)$ and $s_iv\not\in W^J$. In this case $s_iv=vs_j$ for some $j\in J$. 
\end{enumerate}
\end{lem}

Let $w_J\in W_J$ denote the longest element of $W_J$, so $w_{[n]}=w_\circ$. For a subset $J'\subset J$, we set $\Sigma^+_{J/J'}=\Sigma^+_J\backslash \Sigma^+_{J'}$, $\Sigma^-_{J/J'}=\Sigma^-_{J}\backslash \Sigma^-_{J'}$ and we denote by $w_{J/J'}$ the longest minimal coset representative of $W_J/W_{J'}$, so $\ell(w_{J/J'})=\ell(w_J)-\ell(w_{J'})$.

\subsection{Formal group laws} 
Following \cite[IV.2]{Si09}, a one-parameter commutative {\em formal group law (FGL)} $F$ over a commutative unital ring $\Rr$ is a formal power series $F(x,y)\in \Rr[[x,y]]$  that satisfies $F(x,F(y,z))=F(F(x,y),z)$, $F(x,0)=x$ and $F(x,y)=F(y,x)$. It can be written as \[F(x,y)=x+y+\sum_{i,j\ge 1}a_{ij}x^iy^j \quad \text{with}\;a_{ij}\in \Rr.\]  Basic examples are the additive FGL  $F_a(x,y)=x+y$, the multiplicative FGL $F_m(x,y)=x+y-xy$ and the universal FGL where the coefficient ring $\Rr$ is the Lazard ring.

\subsection{Hyperbolic formal group laws}
Consider an elliptic curve given in Weierstrass coordinates by the equation \[ y^2+\mu_1 xy=x^3+\mu_2 x^2 \] over $\Rr=\Z[\mu_1,\mu_2]$. Observe that it is a generic equation for a singular elliptic curve (e.g. see the proof of \cite[III,~Proposition~2.5]{Si09}). The associated formal group law is given by
\begin{equation}\label{hypfgl}
F_{\mu_1,\mu_2}(x,y)=\tfrac{x+y-\mu_1 xy}{1+\mu_2 xy}.
\end{equation}
Such a formal group law will be called hyperbolic (cf.~\cite[\S2.2]{LZ} and \cite[Example 2.2]{CZZ1}).

One of the key results of cobordism theory says that the following are in bijection: formal group laws over $\Rr$ and {\em Hirzebruch genera} with values in $\Rr$. The FGL $F_{\mu_1,\mu_2}$ corresponds to the 2-parameter Todd genus $T_{\al,\be}$, where $\mu_1=\al+\be$ and $\mu_2=-\al\be$. The latter was introduced and studied in~\cite{Kr74} (see also~\cite{BH90}). The particular cases of $T_{\al,\be}$ are the celebrated Hirzebruch genera: the Todd genus $\mu_2=0$, the signature $\mu_1=0$, and the Eulerian characteristic $\al=\be$. Observe that the only singular point of our curve is a cusp in the Eulerian case, and it is a node otherwise.

The fundamental result by Quillen and Levine-Morel~\cite{LM} establishes a correspondence between FGLs and algebraic oriented cohomology theories. If $\hh$ denotes the oriented cohomology theory corresponding to $F_{\mu_1,\mu_2}$ with the coefficient ring $\Rr=\hh(\pt)$, then the Chow group $\hh=CH$ (or usual singular cohomology $H(X,\mathcal{O}_X)$ over $\mathbb{C}$) corresponds to $F_a=F_{0,0}$, and the $K$-theory $\hh=K_0$ corresponds to $F_m=F_{1,0}$. 

\subsection{The generic hyperbolic formal group law}
In the present paper we will deal with the generic version of $F_{\mu_1,\mu_2}$, that is the formal group law $F_t=F_{\mu_1,\mu_2}$ over the coefficient ring $\Rr=\Z[t,t^{-1},(t+t^{-1})^{-1}]$, $t\neq 0$ parametrized as $\al=\tfrac{t}{t+t^{-1}}$ and $\be=\tfrac{t^{-1}}{t+t^{-1}}$, hence
\begin{equation}\label{thypfgl}
\mu_1=1 \text{ and }\mu_2=-\tfrac{1}{(t+t^{-1})^2}\,;
\end{equation}
for simplicity, we denote $u:=-\mu_2=(t+t^{-1})^{-2}$. Such a formal group law together with the respective cohomology theory will be called {\em generic hyperbolic}.

Observe that as $t\to 1$, $F_t(x,y)\to F_1(x,y)=\tfrac{x+y-xy}{1-xy/4}$ over $\Z\left[\tfrac{1}{2}\right]$, which corresponds to the Eulerian genus. If $t\to 0$, then $\mu_2\to 0$, so $F_t(x,y) \to x+y-xy$ over $\Z$, which corresponds to the Todd genus ($K$-theory). The latter suggests that the cohomology corresponding to $F_t$ shares more common properties with $K$-theory than with the Chow group or singular cohomology.

\subsection{Equivariant oriented cohomology}
Let $G$ be a split semi-simple linear algebraic group over a field $k$ of characteristic~$0$ with a split maximal torus $T$ contained in a Borel subgroup~$B$. Consider the associated $T$-equivariant oriented cohomology $\hh_T$ defined on the category of smooth algebraic $T$-varieties following~\cite{HM13} (see also~\cite[\S2]{CZZ2}). By the main result of~\cite{CZZ2}, the $T$-equivariant cohomology ring $\hh_T(G/B)$ of the variety of Borel subgroups $G/B$ can be identified with the dual $\DF^\star$ of the so-called {\em formal affine Demazure algebra} $\DF$. Moreover, in the generic hyperbolic case $F=F_t$, the algebra $\DF$ contains the classical {\em Iwahori-Hecke algebra} $H$ (see~\cite[\S9]{CZZ1} and~\cite{Le15}). We briefly recall the definition of the algebra $\DF$ and the details of the above identification.

\subsection{Twisted group algebras}
First, we define the {\em formal group algebra}
\[
\Ss=\Rr[[\La]]_F:=\Rr[[x_\la]]_{\la\in \La}/(x_0,\; x_{\la+\nu}-F(x_{\la},x_{\nu}))\,,
\]
cf. \cite[Definition~2.4]{CPZ}; we have $\Ss\simeq \hh_T(\pt)$. Let $\Qq$ be the localization of $\Ss$ at all roots, i.e. $\Qq=\Ss\left[\tfrac{1}{x_\al}\mid \al\in \Sigma\right]$. The action of $W$ on $\La$ extends to actions on $\Ss$ and on $\Qq$. Let $\SW$ and $\QW$ denote the respective twisted group algebras of $W$ spanned by $\de_w$, $w\in W$, that is
\[
\SW=\Ss \rtimes_{\Rr} \Rr[W]\text{ and } \QW=\Qq\rtimes_{\Rr} \Rr[W]
\]
and multiplication is given by $\de_w q=w(q)\de_w$, $q\in \Qq$, $w\in W$; we view $\Ss\subset \SW, \Qq\subset \QW$ via the map $q\mapsto q\de_e$.

\subsection{Formal affine Demazure algebras}
Let $\DF$ be the subalgebra of $\QW$  generated by $\Ss$ and by the {\em push-pull elements} $Y_\al=\tfrac{1}{x_{-\al}}+\tfrac{1}{x_\al}\de_{s_\al}$ (or, equivalently, by the {\em Demazure elements} $X_\al=\tfrac{1}{x_\al}\de_{s_\al}-\tfrac{1}{x_\al})$ for all simple roots $\al$. We denote $X_i=X_{\al_i}$ and $Y_i=Y_{\al_i}$. Observe that for the Demazure elements we use the same sign as in \cite[Example~2.3]{LZ}. There is a canonical action of $\QW$ on $\Qq$, defined by 
\begin{equation}\label{actdfs}
(q\de_w) \cdot p=qw(p), ~p,q\in \Qq, w\in W\,,
\end{equation}
which reduces to an action of $\DF$ on $\Ss$. 

If $I_w=(i_1,\ldots,i_k)$ is a reduced sequence of $w$, i.e. $w=s_{i_1}s_{i_2}\ldots s_{i_k}$, denote $I_w^{-1}=(i_k,\ldots,i_1)$, and set $X_{I_w}=X_{i_1}X_{i_2}\cdots X_{i_k}$, $Y_{I_w}=Y_{i_1}Y_{i_2}\cdots Y_{i_k}$. These products depend on the choice of $I_w$ unless $F$ is the additive or a multiplicative formal group law.  By the results of \cite{HMSZ} and \cite{CZZ}, the elements $\{Y_{I_w}\}_{w\in W}$ (resp. $\{X_{I_w}\}_{w\in W}$) form an $\Ss$-basis of $\DF$. Moreover, we know the complete set of relations in $\DF$ (see \cite[Proposition~5.8]{HMSZ} for a general $F$ and \cite[Example~4.12]{Le15} in the hyperbolic case). For instance, the quadratic relations
\begin{equation}
\label{eq:square}X_\al^2=-\kappa_\al X_\al, \quad Y_\al^2=\kappa_\al Y_\al\,, \quad \text{where }\kappa_\al=\tfrac{1}{x_\al}+\tfrac{1}{x_{-\al}}\in \Ss\,,
\end{equation}
will play an essential role in the sequel.

\subsection{The moment map}
Finally, consider the left $\Ss$-module and the left $\Qq$-module duals 
\[
\SWd =Hom_\Ss(\SW,\Ss), \quad \QWd =Hom_\Qq(\QW,\Qq)\; \text{ and }\;\DFd:=Hom_{\Ss}(\DF, \Ss).
\]
Let $\{f_v\}_{v\in W}$ be the  canonical bases of $\SWd$ and $\QWd$, i.e., $f_v(w)=\de_{v,w}$, $v,w\in W$. There is a coordinate-wise product on $\QWd$, defined by $(pf_v)(qf_w)=pq\de_{v,w}f_{v}$ where the multiplicative identity is given by $\unit=\sum_{v\in W}f_v$. The ring $\QWd$ is isomorphic to  localization of the cohomology of the $T$-fixed points set \[ \SWd=Hom(W,\Ss)\cong \hh_T(W) \] and  $\DFd\cong \hh_T(G/B)$ by \cite[Theorem 8.11]{CZZ2}. Under these isomorphisms the natural embedding $\DF^\star \hookrightarrow \SWd \hookrightarrow \QWd$ coincides with the so-called moment map (see \cite[\S10]{CZZ1}). 

\subsection{The parabolic case}
In the present paper we will mostly deal with the parabolic version of the above construction (see \cite{CZZ1} and \cite{CZZ2} for details). That is, for any $J\subset [n]$ we consider a natural projection $p_J\colon \QW \to \Qq_{W/W_J}$,  $\de_w\mapsto \de_{\overline{w}}$, of \cite[\S11]{CZZ1}; here $\Qq_{W/W_J}$ is spanned by $\de_{\overline{w}}$ for $\overline{w}\in W/W_J$. Let $\DFP{J}$ denote the image of $\DF$ in $\Qq_{W/W_J}$. Observe that there is a coproduct on $\Qq_{W/W_J}$ which restricts to $\DFP{J}$. Let $\DFPd{J}=Hom_S(\DFP{J}, \Ss)$ be the dual. By \cite[Lemma~15.1]{CZZ1}, the embedding $p_J^\star\colon \DFPd{J}\hookrightarrow \DFd$ identifies $\DFPd{J}$ with $(\DFd)^{W_J}$, where $W_J$ acts on $\DFd$ via the Hecke $\bullet$-action (see below). Finally, the main result of \cite{CZZ2} says that $\DFPd{J}$ is isomorphic to $\hh_T(G/P_J)$, where $P_J$ is a standard parabolic subgroup corresponding to $J$ and $G/P_J$ is the respective projective homogeneous variety.

\section{Two Hecke actions}\label{sec:twac}
In this section we give the central definition of this paper, i.e., the right Hecke action. 
\subsection{Left and right Hecke actions}\label{subset:two}
We first recall several definitions and facts.
From \cite[Lemma 10.3]{CZZ1} and \cite[Corollary 6.4]{CZZ2} the class of the identity $T$-fixed point in $\DF^\star\cong \hh_T(G/B)$ is given by 
\begin{equation}\label{xpi}
\pt=x_\Pi f_e, \text{ where }x_\Pi=\prod_{\al\in \Sigma^-}x_{\al}\,.
\end{equation}

Following \cite[\S13]{CZZ1} there is an anti-involution $\iota$ of $\QW$ defined by
\begin{equation}\label{eq:iotaonQW}
\iota(q\de_w)=\de_{w^{-1}}q\tfrac{w(x_\Pi)}{x_\Pi}.
\end{equation}
Its restriction to $\Qq$ is the identity map, and it restricts to an involution of $\DF$ that maps $X_i$ to $X_i$  (and hence $Y_i$ to $Y_i$). Thus, we have $\iota(X_I)=X_{I^{-1}}$ and $\iota(Y_I)=Y_{I^{-1}}$, where $I^{-1}$ is the sequence obtained from $I$ by reversing the order.  

Following \cite[\S4]{CZZ1} there is an action of $\QW$ on $\Qq$, defined by 
\[
(z\bullet f)(z')=f(z'z), \quad z,z'\in \QW, f\in \QWd.
\]
From \cite[Corollary 12.1]{CZZ1}, it restricts to an action of $\DF$ on $\DFd$.  Moreover, we have the following result.

\begin{thm} \cite[Theorem 10.13]{CZZ1} \label{thm:DcFdbasis} Via the $\bullet$-action,  $\DFd$ (resp. $\QWd$) is a free $\DF$-module (resp. $\QW$) of rank $1$, and $\pt$ is their basis, i.e., 
\[
\DFd=\DF\bullet \pt, \quad \QWd=\QW\bullet \pt.
\]
\end{thm}

So the elements of  $\DFd$ (resp. $\QWd$) can be uniquely written as $z\bullet \pt$ with $z\in \DF$ (resp. $z\in \QW$). For example, it follows from  \cite[Lemma 4.2]{CZZ1} that 
\begin{equation}\label{eq:zpoint}
qf_v=v^{-1}\left(\tfrac{q}{x_\Pi}\right)\de_{v^{-1}}\bullet \pt, \quad q\in \Qq. 
\end{equation}

We now introduce a key notion of the present paper. 

\begin{dfn}Consider a right action of $\QW$ on $\QWd$, defined as follows:
\begin{gather*}
(z'\bullet \pt)\star z:=z'z\bullet \pt, \quad z, z'\in \QW.
\end{gather*}
Observe that it restricts to a right action of $\DF$ on $\DFd$. Moreover, if composed with the involution $\iota$, it gives left actions of $\DF$ on $\DFd$ and $\QW$ on  $\QWd$, respectively, which will be denoted by $\odot$. That is, 
\begin{equation}\label{eq:odot}
z\odot (z'\bullet pt):=(z'\bullet \pt)\star \iota(z)=z'\iota(z)\bullet \pt, \quad z, z' \text{ belongs to } \DF \text{ or } \QW.
\end{equation}
We call $\bullet$ the {\em left Hecke action} and  $\odot$ the {\em right Hecke action}, respectively. 
\end{dfn}

By definition the left Hecke action is $\Qq$-linear, but the right Hecke action is not. 
\begin{rem} (i) The $\odot$-action depends on the choice of a class of a $T$-fixed point $\pt$ of $\DFd$. If one chooses a class of a different $T$-fixed point from \cite[Theorem~10.13]{CZZ1}, one gets a different version of the $\odot$-action. 

(ii) The terminology is slightly misleading as both the left and right Hecke actions are both left actions (which will be implicitly used throughout). Indeed, both actions are induced by the multiplication on $\DF$: the left Hecke action comes from multiplication from the left in $\DF$, and the right Hecke action comes from multiplication from the right. From the geometric point of view, there is an isomorphism $\hh_T(G/B)\cong \hh_G(G/B\times G/B)$, where the right hand side has a convolution product that defines the product in $\DF$.  So convolution from the left gives the $\bullet$-action, and convolution from the right gives the $\odot$-action. 
\end{rem}

\begin{lem}\label{actcommute} We have
\begin{enumerate}
\item[{\rm (i)}] The left and right Hecke actions commute with each other. 
\item[{\rm (ii)}]  $z\odot\pt=\iota(z)\bullet \pt$. In particular, $X_I\odot \pt=X_{I^{-1}}\bullet \pt$ and $Y_I\odot \pt=Y_{I^{-1}}\bullet \pt$.
\item[{\rm (iii)}] Viewing $\DFd$ as a left $\DF$-module via the $\bullet$-action, there is an isomorphism of rings \[\pi\colon \DF\to \End_{\DF}(\DFd), \quad\pi(z)\colon f\mapsto z\odot f.\]
\end{enumerate}
\end{lem}
\begin{proof}(i) and (ii). Both follow from the definition of the $\odot$-action.

(iii). It follows from part (i) that $\pi$ is well defined. Since $\iota$ is an anti-isomorphism and the $\bullet$-action  is effective,  $\pi$ is injective. Moreover, given any $\phi\in \End_{\DF}(\DFd)$, by Theorem~\ref{thm:DcFdbasis}, $\phi$ is determined by $\phi(\pt)$, which can be written as $z\bullet \pt, z\in \DF$, then $\phi(\pt)=z\bullet \pt=\iota(z)\odot \pt=\pi(\iota(z))(\pt)$, so $\phi=\pi(\iota(z))$. Therefore, $\pi$ is an isomorphism.
\end{proof}

\begin{lem}\label{lem:Heckeactions}The two Hecke actions satisfy
\[
(p\de_w)\bullet (qf_v)=qvw^{-1}(p)f_{vw^{-1}}, ~(p\de_w)\odot (qf_v)=pw(q)f_{wv}, ~p,q\in \Qq, w,v\in W.
\]
\end{lem}
\begin{proof}The first identity is proved in \cite[Lemma 4.2]{CZZ1}. For the second one, we have
\begin{align*}
(p\de_w)\odot (qf_v)\overset{\eqref{eq:zpoint}}=&(p\de_w)\odot \left(v^{-1}\left(\tfrac{q}{x_\Pi}\right)\de_{v^{-1}}\bullet \pt\right)\overset{\eqref{eq:odot}}=\left(v^{-1}\left(\tfrac{q}{x_\Pi}\right)\de_{v^{-1}} \iota(p\de_w)\right)\bullet \pt\\
\overset{\eqref{eq:iotaonQW}}=&\left(v^{-1}\left(\tfrac{q}{x_\Pi}\right)\de_{v^{-1}}\de_{w^{-1}}p\tfrac{w(x_\Pi)}{x_\Pi}\right)\bullet \pt
\overset{\sharp_1}=\left(\de_{v^{-1}w^{-1}}\tfrac{pw(qx_\Pi)}{w(x_\Pi )x_\Pi}\right)\bullet (x_\Pi f_e)\\
\overset{\sharp_2}=&\de_{v^{-1}w^{-1}}\bullet (pw(q)f_e)=pw(q)f_{wv}\,,
\end{align*}
where $\sharp_1$ follows from the product in $\QW$, and $\sharp_2$ follows from the first identity of this lemma. 
\end{proof}

\begin{lem}
We have
\begin{itemize}
\item[{\rm (i)}]  $Y_\al\odot \left(\sum_{v}p_v f_v\right)=\sum_v \left(\tfrac{p_v}{x_{-\al}}+s_\al\left(\tfrac{p_{s_\al v}}{x_{-\al}}\right)\right)f_v$,
\item[{\rm (ii)}]  $X_\al\odot \left(\sum_{v}p_v f_v\right)=\sum_v\tfrac{s_\al(p_{s_\al v})-p_v}{x_\al}f_v$.
\end{itemize}
\end{lem}

\begin{proof}
We just check the first identity. We have 
\[
Y_\al \odot pf_v=\left(\tfrac{1}{x_{-\al}}+\tfrac{1}{x_\al}\de_\al\right)\odot pf_v=\tfrac{p}{x_{-\al}}f_v+s_\al\left(\tfrac{p}{x_{-\al}}\right)f_{s_\al v}\,. \qedhere
\]
\end{proof}

\subsection{Hecke actions and the characteristic map}
We now look at the behavior of the two  Hecke actions with respect to the characteristic map. Consider the equivariant characteristic map \[c_\Ss\colon \Ss \to \DFd \simeq \hh_T(G/B)\text{ given by }s\mapsto s\bullet \unit\] (or, equivalently, by evaluation of the respective operator from $\DF$ at $s\in \Ss$, see \cite[\S11]{CZZ}). After tensoring $\DFd$ with the augmentation $\epsilon\colon \Ss\to \Rr$, it turns into the classical characteristic map of ordinary cohomology $c_\Rr\colon \Ss \to \epsilon\DFd\simeq \hh(G/B)$. Consider the Borel map 
\[
\rho\colon \Ss\otimes_{\Ss^W} \Ss \to \DFd\;\text{ given by }\;q_1\otimes q_2 \mapsto q_1c_S(q_2)=q_1(q_2\bullet \unit)=q_2\bullet (q_1\unit).
\] 
By \cite[Theorem~11.4]{CZZ} it is an isomorphism if the characteristic map $c_\Rr$ is surjective (for example, if the torsion index is $1$, or when $F=F_{1,0}$ is the multiplicative formal group law). 

Using the canonical action \eqref{actdfs} of $\DF$ on $\Ss$, we obtain the following result.

\begin{lem} For any $z\in \DF$, we have the following commutative diagrams.
\[
\xymatrix{\Ss\otimes_{\Ss^W} \Ss \ar[r]^-{\rho}\ar[d]_{(z\cdot\_)\otimes 1} & \DFd \ar[d]^{z\odot\_} && \Ss\otimes_{\Ss^W}\Ss\ar[r]^-{\rho}\ar[d]_{1\otimes (z\cdot\_)} & \DFd\ar[d]^{z\bullet\_}\\
\Ss\otimes_{\Ss^W}\Ss \ar[r]^-{\rho} & \DFd && \Ss\otimes_{\Ss^W}\Ss \ar[r]^-{\rho} & \DFd}
\]
\end{lem}

\begin{proof}
By direct computation using Lemma~\ref{lem:Heckeactions} and the definition of $\rho$. 
\end{proof}

\begin{rem}
Using the element $Y_\Pi$ introduced in the next section, the lemma can be also deduced from the following identities: 
\[
\iota(Y_\Pi)=Y_\Pi, \quad Y_\Pi\bullet \pt=\unit\quad \text{ and } \quad
zpY_\Pi=(z\cdot p)Y_\Pi\; \text{  for any }z\in \DF,\; p\in \Ss.
\]
Moreover, using \eqref{prop2} below, we have
\begin{gather*}
Y_\al\odot \unit=Y_\al\odot Y_\Pi \bullet \pt=Y_\Pi Y_\al \bullet \pt=\kappa_\al \unit, \quad Y_\al\bullet \unit=Y_\al Y_\Pi\bullet \pt=\kappa_\al \bullet \unit,\\ X_\al\odot \unit=Y_\Pi X_\al\bullet \pt=0, 
 \quad X_\al \bullet \unit=X_\al Y_\Pi \bullet \unit=0.
\end{gather*}
\end{rem}

The following Proposition was proved in \cite[page~253]{Br97} for equivariant Chow groups, and in \cite[Lemma 13.4]{CPZ} for ordinary oriented cohomology.

\begin{prop} (cf. also \cite[Lemma~13.3]{CZZ1}) For any $z\in \QW,\; q\in \Qq$, we have 
\[
A_\Pi((z\odot \pt)c_{\Ss}(q))=(z\cdot q)\unit,
\]
where $A_\Pi= Y_\Pi\bullet\_$ is the pushforward to the point map (see the definition in the next section). 
\end{prop}

\subsection{Hecke actions on cohomology}\label{heckeactions}
We now study the Hecke actions on the cohomology ring $\hh_T(G/P_J)$, where $P_J$ is the parabolic subgroup corresponding to a subset $J\subset\Pi$. 

We first look at the $\bullet$-action. Given a subset $J'\subset J$ we define
\[
x_{J/J'}=\prod_{\al\in \Sigma^-_{J/J'}}x_\al,\quad x_J=x_{J/\emptyset}.
\]
Note that the definition of $x_{\Pi}$ in~\eqref{xpi} is a special case of the one above. Given a set of left coset representatives $W_{J/J'}$ of $W_J/W_{J'}$, we define a push-pull element 
\[
Y_{J/J'}:=\left(\sum_{w\in W_{J/J'}}\de_w\right)\tfrac{1}{x_{J/J'}}, \quad Y_J=Y_{J/\emptyset}=\left(\sum_{w\in W_J}\de_w\right)\tfrac{1}{x_J}.
\]
Clearly, if $J=\{i\}$, we recover the definition of $Y_i$. The $\bullet$-action of $Y_{J/J'}$ on $\QWd$ is denoted by $A_{J/J'}$. It follows from \cite[Corollary~12.1]{CZZ1} that $Y_J\in \DF$, so $A_J$ is an endomorphism of $\DFd$, which corresponds to the composition  of the pushforward and pullback maps in cohomology 
\[
A_J\colon \hh_T(G/B)\to \hh_T(G/P_J)\to \hh_T(G/B).
\]  

Let $J''\subset J'\subset J$, and let $W_{J/J'}$ and $W_{J'/J''}$ be some sets of  coset representatives of $W_J/W_{J'}$ and $W_{J'}/W_{J''}$, respectively. Let $W_{J/J''}=W_{J/J'}W_{J'/J''}$, which is a set of coset representatives of $W_J/W_{J''}$. We then have (see \cite[Lemma~5.7]{CZZ1}) 
\begin{equation}\label{prop1} 
Y_{J/J''}=Y_{J/J'}Y_{J'/J''}\quad\text{ and }\quad A_{J/J''}=A_{J/J'}\circ A_{J'/J''}.
\end{equation}
If $j\in J$, then (see \cite[Lemma~5.9]{CZZ1}) 
\begin{equation}\label{prop2} 
Y_jY_J=\kappa_jY_J, ~Y_JY_j=Y_J\kappa_j \quad\text{ and }\quad X_jY_J=Y_JX_j=0.
\end{equation}

Define $\DFPd{J}=A_J(\DFd)$. It follows from \cite[Corollary 14.6, Lemma 15.1]{CZZ1} and \cite[Theorem 8.11]{CZZ2} that 
\[
\DFPd{J}=(\DFd)^{W_J}\cong \hh_T(G/P_J)\,.
\]
Here the superscript $W_J$ means the $W_J$-invariant subset via the $\bullet$-action of $W_J\subset W\subset \DF$ on $\DFd$. Moreover,  they also imply that  the endomorphism $A_{J/J'}$ of $\QWd$ induces a surjective map $\DFPd{J'}\to \DFPd{J}$, which corresponds to the pushforward map $\hh_T(G/P_{J'})\to \hh_T(G/P_J)$.

\begin{rem}
The definition of $A_{J/J'}$ depends on the choice of the coset representatives of $W_J/W_{J'}$, but its restriction to $\DFPd{J'}$ does not (see \cite[Lemma~6.5]{CZZ1}).
\end{rem}

The element $A_J(\pt)\in \DFPd{J}$ corresponds to the class of the identity point in $\hh_T(G/P_J)$, so we denote it by $\pt_J$. Based on definitions, this can be expressed more explicitly as follows:
\begin{equation}\label{calcbs} 
\pt_J=x_{\Pi/J}\sum_{w\in W_J} f_w\,.
\end{equation}
By using $w(x_{\Pi/J})=x_{\Pi/J}$ for $w\in W_J$, we have $\iota(Y_J)=Y_J$, so
\begin{equation}\label{bullodot}
\pt_J=Y_J\bullet\pt=Y_J\odot\pt\,.
\end{equation}

\begin{lem}
The $\odot$-action of $\DF$ on $\DFd$ restricts to an action on $\DFPd{J}$, making the latter into a $\DF$-module. 
\end{lem}

\begin{proof}
This follows from the definition of $\DFPd{J}$ and Lemma~\ref{actcommute}.(i).
\end{proof}

By definition, the $\odot$-action of $\DF$ on $\DFPd{J}$  can be expressed as follows:
\[
\DF\to \DFPd{J}\,, \;\;z\mapsto z\odot \pt_J=z\odot A_J(\pt)=A_J(z\odot\pt)=A_J(\iota(z)\bullet \pt)\,.
\]
Thus, the {\em parabolic Bott-Samelson classes} inside $\DFPd{J}\cong \hh_T(G/P_J)$ are determined by 
\begin{equation}\label{pbsy}
\mY{J}{w}:=Y_{I_w}\odot\pt_J=A_J(Y_{I^{-1}_w}\bullet \pt)\,, \quad w\in W^J\,.
\end{equation}
Similarly, one can look at the classes determined by the $X$-operators, namely at
\begin{equation}\label{pbsx}
\mX{J}{w}:=X_{I_w}\odot\pt_J=A_J(X_{I_w^{-1}}\bullet \pt)\,, \quad w\in W^J\,.
\end{equation}
We have $\mX{}{e}=\mY{}{e}=\pt$ and $\mX{J}{e}=\mY{J}{e}=\pt_J$. It follows from \cite[Theorem~14.3]{CZZ1} that the sets $\{\mX{J}{w}\}_{w\in W^J}$ and $\{\mY{J}{w}\}_{w\in W^J}$ are  bases of $\DFPd{J}$. Therefore,  the following theorem holds:

\begin{thm} \label{thm:main}
All parabolic Bott-Samelson classes $\zeta_{I_w}^J$, $w\in W^J$, in $\DFPd{J} \simeq \hh_T(G/P)$ can be obtained by applying the operators $Y_{i}\odot\_$, $i\in [n]$ to the class of the point $\mY{J}{e}=\pt_J$. The same holds for the classes $\hat\zeta_{I_w}^J$ if we use the operators $X_i\odot\_$ instead. 
\end{thm}

\begin{rem}\label{rempbs} (i) By \eqref{pbsy} and \eqref{pbsx}, the parabolic Bott-Samelson classes can be calculated both via the $\odot$- and the $\bullet$-actions. The advantage of the $\odot$-action is that it leads to a recursive computation which stays inside $\DFPd{J}$. More precisely, given a reduced sequence  $I_w=(i_1,i_2,\ldots,i_k)$ of $w\in W^J$, we know from Lemma~\ref{lem:parw} that $v:=s_{i_1}w<w$ is in $W^J$; so $I_{v}:=(i_2,\ldots,i_k)$ is a reduced sequence of $v$, and we have
\begin{equation}\label{extopclass}
\mY{J}{w}=Y_{i_1}\odot\mY{J}{v}\,,\quad \mX{J}{w}=X_{i_1}\odot\mX{J}{v}\,.
\end{equation}

(ii) Also by \eqref{pbsy} and \eqref{pbsx}, the parabolic Bott-Samelson classes are obtained from the ones corresponding to $G/B$ via the pushforward map, namely
\[
\mY{J}{w}=A_J(\mY{}{w})\,,\quad \mX{J}{w}=A_J(\mX{}{w}),\quad w\in W^J\,.
\]
On the other hand, in ordinary cohomology and $K$-theory, it is well-known that the pullback map sends Schubert classes to Schubert classes, see e.g. \cite[\S10.6]{fulyt}. This can be seen easily in our algebraic setting. Indeed, for $F=F_a$ or $F_m$, it follows from similar argument as in Proposition~\ref{prop:gammaJ}.(ii) that $Y_J=Y_{w_J}$. So 
\[
\zeta_{w}^J=A_J(\zeta_w)=Y_{w_J}\bullet Y_{w^{-1}}\bullet \pt=Y_{(ww_J)^{-1}}\bullet\pt=\zeta_{ww_J},\quad w\in W^J.
\] 
For general $F$, we may not have $Y_J=Y_{w_J}$. However, see Theorem~\ref{thm:C} below for a similar situation.
\end{rem}

\subsection{Hecke actions on parabolic Bott-Samelson classes}\label{subsec:BS}
We now study the action of $A_{J/J'}$ on the classes $\zeta_{I_w}^J$ and $\hat\zeta_{I_w}^J$, with the goal of extending Remark~\ref{rempbs}~(ii). For $J'\subset J$, we have $W^J\subset W^{J'}$, and by definition, the $\Ss$-linear homomorphism $A_{J/J'}$ induced by $Y_{J/J'}$ commutes with the $\odot$-action of $\DF$ on $\DFPd{J}$.

\begin{lem}\label{lem:partialD}
Let $\{I_w\}_{w\in W}$ be a set of reduced sequences. Let $J'\subset J$, and $v\in W^{J'}$, then 
\[
A_{J/J'}(\mX{J'}{v})=\left\{\begin{array}{ll}\mX{J}{v}, & \text{if }v\in W^J;\\
\underset{u\in W^{J},  \ell(u)<\ell(v)}\sum d_u \mX{J}{u}, &\text{if }v\not\in W^J.\end{array}\right. 
\]
The same conclusion holds if one replaces $\mX{J'}{v}$ by $\mY{J'}{v}$. 
\end{lem}

\begin{proof} 
By~\eqref{prop1}, we have
\[
A_{J/J'}(\mX{J'}{v})=A_{J/J'}\circ A_{J'}( X_{I_v}\odot\pt)=A_{J}(X_{I_{v}}\odot \pt)=A_J(X_{I_v^{-1}}\bullet \pt).
\]
If $v\in W^J$, it follows that $A_{J/J'}(\mX{J'}{v})=\mX{J}{v}$.

If $v\not\in W^{J}$, then $v=uw$ with $u\in W^J, w\in W_J$ and $\ell(u)<\ell(v)$. Hence, the concatenation of $I_u$ and $I_w$ is another reduced sequence of $v$. We prove the conclusion by induction on $\ell(u)$. If $\ell(u)=0$, then $v=w$ and $A_J(X_{I_w^{-1}}\bullet \pt)=\de_{w,e}\mX{J}{e}$, by~\eqref{prop2}. So the conclusion follows.

Now suppose $\ell(u)\ge 1$. By~\cite[Lemma~7.1]{CZZ}, we have 
\[
X_{I_v}-X_{I_{u}}X_{I_w}=\sum_{v'\in W,\ell(v')<\ell(v)}d_{v'}X_{I_{v'}}. 
\]
So
\[
A_J(X_{I_v}\odot \pt)=A_J(X_{I_u}X_{I_w}\odot\pt)+\sum A_J((d_{v'}X_{I_{v'}})\odot\pt). 
\]
By induction, 
\[
A_J((d_{v'}X_{I_{v'}})\odot \pt)=A_J(X_{I_{v'}^{-1}}d_{v'}\bullet \pt)=d_{v'}A_J(X_{I_{v'}^{-1}}\bullet \pt)
\]
is a linear combination of $\mX{J}{u'}$ with $u'\in W^J$ and $\ell(u')\le\ell(v')<\ell(v)$. Moreover, 
\[
A_J(X_{I_u}X_{I_w}\odot \pt)=A_J(X_{I_w^{-1}}X_{I_u^{-1}}\bullet \pt)=\de_{w,e}\mX{J}{u}\,,
\] by~\eqref{prop2}. So the conclusion follows.

The statement for $\mY{J}{v}$ can be proved similarly. 
\end{proof}

We will now state a more explicit version of Lemma~\ref{lem:partialD} in a special case. 

\begin{dfn}For $J'\subset J$, we say a set of reduced sequences $\{I_w\}_{w\in W}$ is $J'\subset J$-compatible if for any $w=v_1v_2v_3$ such that $v_1\in W^{J}, v_2\in W_{J}, v_3\in W_{J'}$ and $v_1v_2\in W^{J'}$, the sequence $I_w$ is the concatenation of $I_{v_1}$, $I_{v_2}$, and $I_{v_3}$ in this order.
\end{dfn}

\begin{lem} If $\{I_w\}_{w\in W}$ is $J'\subset J$-compatible, then for any $v_1\in W^J, v_2\in W_J$ such that $v_1v_2\in W^{J'}$,  we have
\[
A_{J/J'}(\mX{J'}{v_1v_2})=\de_{v_2,e}\mX{J}{v_1}. 
\]
If $\kappa_\al$ is a constant $\kappa$ for any $\al$, then 
\[
 A_{J/J'}(\mY{J'}{v_1v_2})=\kappa^{\ell(v_2)}\mY{J}{v_1}. 
\]
\end{lem}
\begin{proof}The proof is similar to that of Lemma~\ref{lem:partialD}. 
\end{proof}
{If $F$ is $F_a$ or $F_m$, that is, in the case of the Chow group or the Grothendieck group, both conditions in the above lemma are satisfied.}

\subsection{Recursive formulas}\label{subsec:recform}
Now assume that $\kappa_\al=\kappa$ is a constant, and that $W^J$ is {\em fully commutative}; this means that for each $w\in W^J$, any two reduced sequences $I$, $I'$ of $w$ can be related by means of only braid relations involving commuting generators. Thus, $X_I={X_{I'}}$ and $Y_I=Y_{I'}$, so we can write $X_w$ instead. Note that, although $w^{-1}$ might not be in $W^J$, it still has the fully commutative property, so $X_{w^{-1}}$ is defined and  $\iota(X_w)=X_{w^{-1}}$. 

\begin{rem} 
For irreducible finite Weyl groups $W$, the fully commutative parabolic quotients $W^J$ were classified in \cite[Theorem~6.1]{St96}. They are the irreducible {\em minuscule quotients}, which are also maximal quotients (so the complement of $J$ consists of a single element); for the explicit list, which includes the type $A$ Grassmannians, see also~\cite{St96}. 
\end{rem}

\begin{prop}\label{lem:grass} 
Under the above hypothesis, we have
\[
X_j\odot \mXG{J}{v}=\left\{\begin{array}{ll} -\kappa \mXG{J}{v} ,& \text{if }\ell(s_jv)< \ell(v), \text{in which case }s_jv\in W^J;\\
\mXG{J}{s_jv} , & \text{if } \ell(s_jv)> \ell(v) \text{ and } s_jv\in W^J;\\
\underset{u\in W^J,\ell(u)\le \ell(v)-1}\sum d^J_{v,u}\mXG{J}{u}  , & \text{if } \ell(s_jv)> \ell(v) \text{ and } s_jv=vs_i \text{ for some }i\in J,\end{array}\right.
\]	
where $d^J_{v,u}\in \Ss$.
\end{prop}

\begin{proof}
By definition we have $X_j\odot \mXG{J}{v}=X_jX_v\odot \pt_J$.

If $\ell(s_jv)< \ell(v)$, then $X_{v}=X_jX_{s_jv}$, and $X_jX_v=X_j^2X_{s_jv}=-\kappa X_jX_{s_jv}=-\kappa X_v$. So $X_j\odot \mXG{J}{v}=-\kappa\mXG{J}{v}$.
	
If $\ell(s_jv)> \ell(v)$ and $s_jv\in W^J$, then Remark~\ref{rempbs}~(i) applies.
	
If $\ell(s_jv)> \ell(v)$ and $s_jv=vs_i$ with $i\in J$, then by \cite[Lemma~7.1]{CZZ},
\[
X_{j}X_{v}-X_{v}X_i=\sum_{w\in W,\ell(w)\le \ell(v)-1}\tilde d_{v,w}X_{I_w}, ~\tilde d_{v,w}^J\in \Ss.
\]
So 
\[
X_jX_v\odot \pt_J=X_vX_i\odot \pt_J+\sum_{w\in W, \ell(w)\le \ell(v)-1} \tilde d_{v,w}X_{I_w}\odot\pt_J.
\]
We have 
\[
X_vX_i\odot \pt_J\overset{\eqref{bullodot}}=X_vX_i\odot Y_J\odot\pt=X_v\odot (X_iY_J)\odot\pt\overset{\eqref{prop2}}=0.
\]
Moreover, we know that 
\[
\tilde d_{v,w}X_{I_w}\odot\pt_J=A_J(\iota(\tilde d_{v,w}X_{I_w})\bullet \pt)=A_J(X_{I_w^{-1}}\tilde d_{v,w}\bullet \pt)=\tilde d_{v,w}A_J(X_{I_w^{-1}}\bullet \pt)
\] 
is a linear combination of $\mXG{J}{u}$ with $u\in W^J$ and $\ell(u)\le \ell(w)\le \ell(v)-1$, by the $\Ss$-linearity of the $\bullet$-action and the basis property of $\DFPd{J}$. 
\end{proof}

Similarly, using \eqref{eq:square} and \eqref{prop2}, we have the following result. 

\begin{prop}\label{lem:Yact} Under the above hypothesis, we have
\[
Y_j\odot \mYG{J}{v}=\left\{\begin{array}{ll} \kappa \mYG{J}{v} ,& \text{if }\ell(s_jv)< \ell(v), \text{in which case }s_jv\in W^J;\\
\mYG{J}{s_jv} , & \text{if } \ell(s_jv)> \ell(v) \text{ and } s_jv\in W^J;\\
\kappa\mYG{J}{v}+\underset{u\in W^J, \ell(u)\le \ell(v)-1}\sum \hat d^J_{v,u}\mYG{J}{u}  , & \text{if } \ell(s_jv)> \ell(v) \text{ and } s_jv=vs_i \text{ for some }i\in J\end{array}\right.
\]	
where $\hat d^J_{v,u}\in \Ss$.
\end{prop}

In the case of an additive $F_a$ or a multiplicative $F_m$ formal group law the proofs of Propositions~\ref{lem:grass} and~\ref{lem:Yact} work for any flag variety $G/P_J$
(as the respective $X_{I_v}$ and $Y_{I_v}$ do not depend on the choice of a reduced sequence $I_v$ of $v$).
So we obtain the following formulas.

\begin{ex}\label{ex:Chowformula} In the case $F=F_a$, we have $X_i=Y_i$, $\kappa_\al=0$, and
	\[
	X_j\odot \naxi{v}=\left\{\begin{array}{ll} 0 ,& \text{if }\ell(s_jv)< \ell(v), \text{in which case }s_jv\in W^J;\\
	 \naxi{s_jv} , & \text{if } \ell(s_jv)> \ell(v) \text{ and } s_jv\in W^J;\\
0  , & \text{if } \ell(s_jv)> \ell(v) \text{ and } s_jv=vs_i \text{ for some }i\in J.\end{array}\right.
	\]	
 This formula was known before to Brion \cite{Br97}, Knutson \cite{Kn03}, and Tymoczko \cite{Ty09}.
\end{ex}

The next two formulas, which correspond to the case of $K$-theory, seem to be new; here we use $\kappa_\al=1$. 

\begin{ex}\label{ex:formulaK} We have
	\begin{align*}
	X_j\odot \nmxiX{v}&=\left\{\begin{array}{ll} -\nmxiX{v}, & \text{if }\ell(s_jv)< \ell(v), \text{in which case }s_jv\in W^J;\\
\nmxiX{s_jv},  & \text{if } \ell(s_jv)> \ell(v) \text{ and } s_jv\in W^J;\\
	0,   & \text{if } \ell(s_jv)> \ell(v) \text{ and } s_jv=vs_i \text{ for some }i\in J;\end{array}\right.\\
	Y_j\odot \nmxiY{v}&=\left\{\begin{array}{ll} \nmxiY{v}, & \text{if }\ell(s_jv)< \ell(v), \text{in which case }s_jv\in W^J;\\
\nmxiY{s_jv} , & \text{if } \ell(s_jv)> \ell(v) \text{ and } s_jv\in W^J;\\
	\nmxiY{v} ,  & \text{if } \ell(s_jv)> \ell(v) \text{ and } s_jv=vs_i \text{ for some }i\in J.\end{array}\right.
	\end{align*}	
\end{ex}

\subsection{An example}\label{secex} We give an example for the calculation of the parabolic Bott-Samelson classes in the case of the hyperbolic formal group law $F_{1,-u}$. Consider the Grassmannian of $2$-planes in ${\mathbb C}^4$, so $W=S_4$, $J=\{1,3\}$, and 
\[
W^J=\{{e},\,s_2,\,s_3s_2,\, s_1s_2,\,s_1s_3s_2,\,s_2s_1s_3s_2\}\,.
\]
This Grassmannian contains a singular Schubert variety, namely the one indexed by $s_1s_3s_2$. Recall that, since $W^J$ is fully commutative, we can index the parabolic Bott-Samelson classes by the elements of $W^J$. 

We summarize the recursive calculation~\eqref{calcbs} of the parabolic Bott-Samelson classes via the $\odot$-action in the following diagram; here the edges are labeled by the simple reflections used in the corresponding left multiplication. 
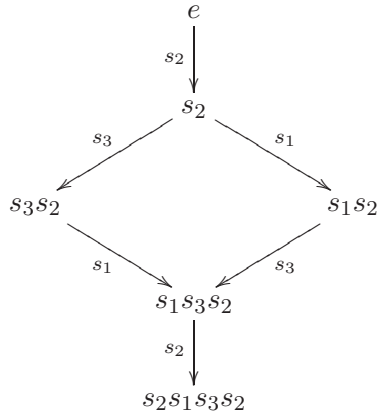
\[ \scriptstyle
{\xymatrix{
&{e}\ar[d]_{s_2}\\
&{s_2}\ar[dl]_{s_3}\ar[dr]^{s_1}\\
s_3s_2\ar[dr]_{s_1}&&s_1s_2\ar[dl]^{s_3}\\
&s_1s_3s_2\ar[d]_{s_2}\\
&s_2s_1s_3s_2
} }
\]
For simplicity, we use the notation $[ij]:=x_{-\al_{ij}}$, where $\al_{ij}$ is the root $\varepsilon_i-\varepsilon_j$. Writing $\zeta_v^J=\sum_{w\in W} q_w f_w$, we note that it is enough to record the coefficients $q_w$ for $w\in W^J$ (because $\zeta_v^J$ can be viewed as elements of $(\DFd)^{W_J}$, see \cite[Lemma 6.1]{CZZ1});  we do this by also using the above diagram. 

Recalling the expression of $\zeta_{e}^J$ in~\eqref{extopclass}, we proceed as follows. 
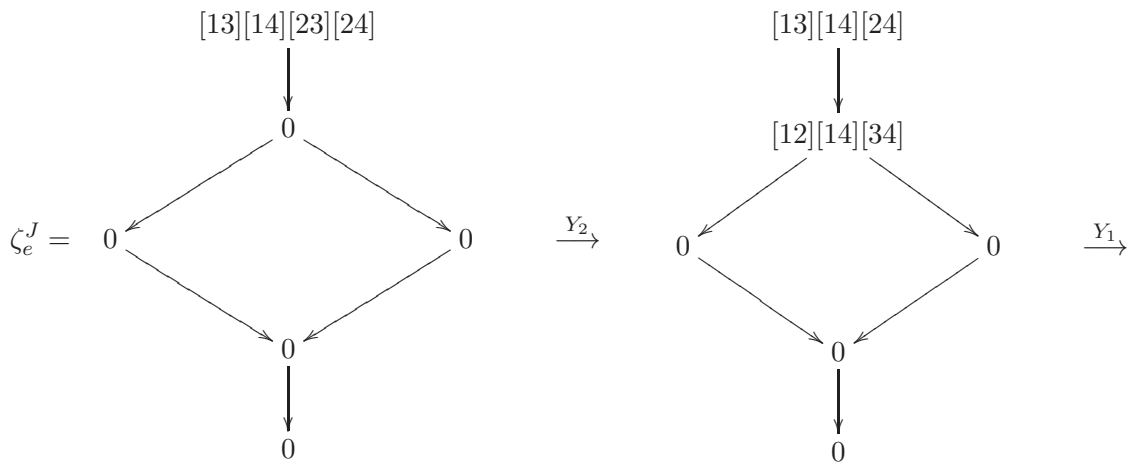
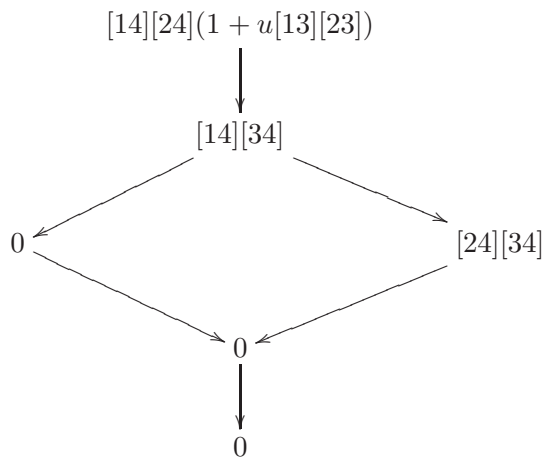
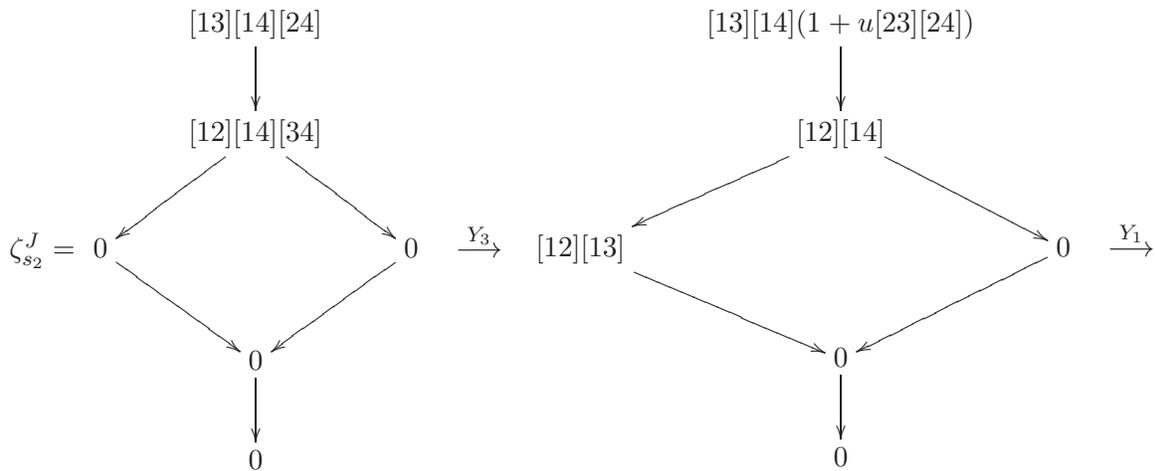
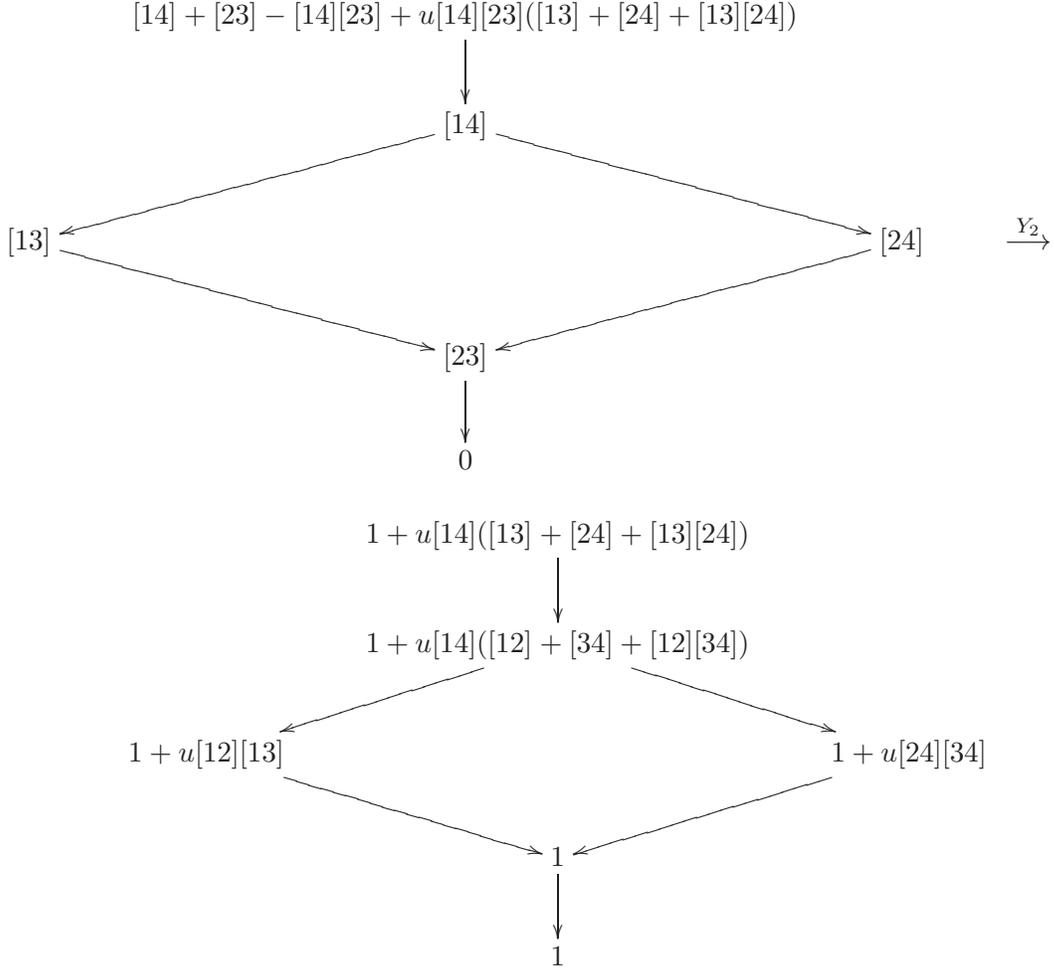
\[\scriptstyle
{\xymatrix{
&&{[13][14][23][24]}\ar[d]\\
&&{0}\ar[dl]\ar[dr]\\
\zeta_{e}^J=\!\!\!\!\!\!\!\!\!\!&{0}\ar[dr]&&{0}\ar[dl] & \stackrel{Y_2}{\longrightarrow}\qquad\\
&&{0}\ar[d]\\
&&{0}
} }
{\xymatrix{
&{[13][14][24]}\ar[d]\\
&{[12][14][34]}\ar[dl]\ar[dr]\\
{0}\ar[dr]&&{0}\ar[dl]& \stackrel{Y_1}{\longrightarrow}\qquad\\
&{0}\ar[d]\\
&{0}
} }\]
\[\scriptstyle
{\xymatrix{
&{[14][24](1+u[13][23])}\ar[d]\\
&{[14][34]}\ar[dl]\ar[dr]\\
{0}\ar[dr]&&{[24][34]}\ar[dl]\\
&{0}\ar[d]\\
&{0}
} }
\]
\[\scriptstyle
{\xymatrix{
&&{[13][14][24]}\ar[d]\\
&&{[12][14][34]}\ar[dl]\ar[dr]\\
\zeta_{s_2}^J=\!\!\!\!\!\!\!\!\!\!\!\!\!&{0}\ar[dr]&&{0}\ar[dl]& \!\!\!\!\!\!\!\!\!\stackrel{Y_3}{\longrightarrow}\;\;\\
&&{0}\ar[d]\\
&&{0}
} }
{\xymatrix{
&{[13][14](1+u[23][24])}\ar[d]\\
&{[12][14]}\ar[dl]\ar[dr]\\
{[12][13]}\ar[dr]&&{0}\ar[dl] &\!\!\!\!\!\!\!\!\!\stackrel{Y_1}{\longrightarrow}\;\;\\
&{0}\ar[d]\\
&{0}
} }
\]
\[\scriptstyle
{\xymatrix{
&{[14]+[23]-[14][23]+u[14][23]([13]+[24]+[13][24])}\ar[d]\\
&{[14]}\ar[ld]\ar[rd]\\
{[13]}\ar[dr]&&{[24]}\ar[dl] &\stackrel{Y_2}{\longrightarrow}\qquad\\
&{[23]}\ar[d]\\
&{0}
} }
\]
\[\scriptstyle
{\xymatrix{
&{1+u[14]([13]+[24]+[13][24])}\ar[d]\\
&{1+u[14]([12]+[34]+[12][34])}\ar[dl]\ar[dr]\\
{1+u[12][13]}\ar[dr]&&{1+u[24][34]}\ar[dl]\\
&{1}\ar[d]\\
&{1}
} }
\]

In the above calculations, we repeatedly used the following identities in the corresponding formal group algebra $\Ss$:
\begin{equation}\label{ident}
\tfrac{x_{\al+\be}}{x_\al}+\tfrac{x_\be}{x_{-\al}}=1+ux_\be x_{\al+\be} \,,\;\;\;\;\;\tfrac{x_{\al+\be}-x_\be}{x_\al}=1-x_\be+ux_\be x_{\al+\be}\,.
\end{equation}
The first identity is part of \cite[Lemma~4.2]{LZ1}; the second one follows in the same way, using the fact that, by the definition of the hyperbolic formal group law, we have
\[
x_{\al+\be}=x_\al+x_\be-x_\al x_\be+ux_\al x_\be x_{\al+\be}\,.
\]
We also used  $k_\al=1$.

In addition, in the calculation of $\zeta_{s_1s_3s_2}$, we used 
{\small \begin{align*} 
\tfrac{[13][14]}{[12]}+\tfrac{[23][24]}{[21]}&=[14]\tfrac{[13]-[23]}{[12]}+[23]\left(\tfrac{[14]}{[12]}+\tfrac{[24]}{[21]}\right)\\
&=[14](1-[23]+u[13][23])+[23](1+u[14][24])\,.
\end{align*}}

\subsection{A generalized Deodhar's resolution}\label{subsec:deod}
We apply the Hecke actions to construct a resolution of $\DFd$ using the algebras $\DFPd{J}$, for $J\subset \Pi$ and an arbitrary formal group law $F$. This generalizes Deodhar's acyclic complex for the Iwahori-Hecke algebra \cite[\S5]{D87}, and also offers a geometric interpretation of it. 

\begin{thm} We have a chain complex of $\DF$-modules (also of free $\Ss$-modules)
\begin{equation}\label{eq:mainseq}
\xymatrix{0\ar[r]& \DFd \ar[r]^-{\partial_0} & \underset{|J|=1}\bigoplus\DFPd{J}\ar[r]^-{\partial_1} & \underset{|J|=2}\bigoplus\DFPd{J}\ar[r]^-{\partial_2} & \cdots \ar[r]^-{\partial_{n-1}}& \DFPd{[n]}\ar[r]^-{\partial_n} & 0}\,,
\end{equation}
where the maps are defined as follows: 
\[
\partial_r=\underset{|J'|=|J|-1=r}\sum\epsilon(J, J')A_{J/J'}\,,
\]
with $\epsilon(J, J')=(-1)^{|\{i'\in J'|i'<i\}|}$ if $J= J'\cup \{i\}$. Moreover, its cohomology $H^r=0$ for $r\ge 1$. 
\end{thm}

\begin{proof}
Although in a different context, the proof that this sequence is a chain complex follows Deodhar's proof of  \cite[Theorem~5.1~(i)]{D87} using the identity in~\eqref{prop1}. The only difference is concerned with the use of Lemma~\ref{lem:partialD}, where in the second case of the corresponding formula ($v\not \in W^J$) the evaluation is just $0$ in Deodhar's case.
\end{proof}

\begin{thm}If $F=F_a$ or $F_m$, then $\ker \partial_0$ is a free module of rank $1$ over $\Ss$, with basis $\{\mX{\emptyset}{w_\circ}\}$ where $w_\circ$ is the longest element of $W$.
\end{thm}

\begin{proof} 
In these cases, $X_{I_v}$ and $Y_{I_v}$ does not depend on the choice of $I_v$, so we can write $X_v$ and $Y_v$ respectively. If $\ell(v)\ge \ell(vs_i)$, then $X_{v^{-1}}=X_iX_{(vs_i)^{-1}}$, and hence, $A_i(X_{v^{-1}}\bullet \pt)=0$.  Let $g=\sum_{v\in W}d_v(X_{v^{-1}}\bullet \pt)\in \ker \partial_0$, then 
\[
0=\partial_0(g)=\sum_{i=1}^n\left(\sum_{v\in W^{\{i\}}}A_i\left(d_v(X_{v^{-1}}\bullet \pt)+d_{vs_i}(X_iX_{v^{-1}}\bullet \pt)\right)\right)=\sum_{i=1}^n\left(\sum_{v\in W^{\{i\}}}d_v\mX{i}{v}\right)\,.
\]
For any $v\neq  w_\circ$, there exists $i$ such that $v\in W^{\{i\}}$, so $d_v=0$ and we obtain $g=d_{w_\circ}(X_{w_\circ}\bullet \pt)$. 
\end{proof}

\begin{ex} 
Let $F=F_a$ and the root datum be of Dynkin type $A_2$. There are four $\DF$-modules, $\DFd, \DFPd{1}, \DFPd{2}, \DFPd{\Pi}\cong \Ss$. Moreover, $\epsilon(1,\emptyset)=\epsilon(2,\emptyset)=\epsilon(\Pi, 2)=1,\epsilon(\Pi, 1)=-1$, so the chain complex is 
\[
\xymatrix{0\ar[r] & \DFd \ar[rr]^-{(A_{1},  A_2)} & & \DFPd{1}\oplus \DFPd{2} \ar[rrr]^-{(-A_{\Pi/1})\oplus A_{\Pi/2}} && & \DFPd{\Pi} \ar[r] & 0.}
\]
We know that $A_{\Pi/1}(\zeta^{1}_v)=0,$ for $v\in W^{1}, v\neq e$ and $A_{\Pi/1}(\zeta^{1}_e)=\zeta^{\Pi}_e$, and similar conclusions hold if one replaces $1$ by $2$, so 
\[
\partial_1\left(\sum_{v\in W^1}d_v^1\zeta^1_v, \sum_{u\in W^2}d_u^2 \zeta^2_u\right)=d^1_e-d_e^2.
\]
So $\im \partial_1=\DFPd{\Pi}$, and 
\[
\ker \partial_1=\left\{\left(\sum_{v\in W^1}d_v^1\zeta^1_v, \sum_{u\in W^2}d_u^2 \zeta^2_u\right)\mid d_e^1=d_e^2\right\}.
\]
For any $z=\sum_{w\in W}d_w\zeta_w$, 
\[
\partial_0(z)=(A_1(z), A_2(z))=\left(\sum_{w\in W^1}d_w\zeta_w^1, \sum_{w'\in W^2}d_{w'}\zeta^2_{w'}\right)\,,
\]
so $\im \partial_0=\ker \partial_1$. Moreover, $\ker \partial_0=\{d\zeta_{w_\circ}\mid d\in \Ss\}\cong \Ss$. 
\end{ex}

\section{Deodhar's parabolic Hecke modules revisited}\label{sec:Deodmult}
The goal of this section is to give a new interpretation of Deodhar's modules over Iwahori-Hecke algebras in terms of the equivariant $K$-theory of flag varieties; a related dictionary is provided. All the results stated here are known, but we provide proofs in the new setup, which give new insights and are sometimes simpler.  

\subsection{The case of a multiplicative formal group law}
In this section we consider only the multiplicative FGL $F_m=x+y-xy$. We apply definitions and results of \S2 and \S3 to the case of $F_m$. We set $R=\Z[t,t^{-1}]$ and $Q_m=R[\La][\tfrac{1}{1-e^\al}|\al\in \Sigma]$. Let $Q_{W,m}=Q_m\rtimes R[W]$ be the twisted group algebra with basis $\de_w^m$ indexed by $w\in W$. We define the divided difference elements
\[
X_i^m=\tfrac{1}{1-e^{-\al_i}}(\de_i^m-1)\in Q_{W,m}, 
\]
and define $X_w^m$ as products of $X_i^m$'s corresponding to reduced sequences.  Let $\DFm$ be the subalgebra generated by $X_i^m$ and $R[\La]$. This is the affine 0-Hecke algebra  defined in \cite[Definition~2.8]{KK90} (see also \cite[Definition~5.7]{CZZ}).  The set $\{X^m_{w}\}_{w\in W}$ is a basis of $\DFm$ as an $R[\La]$-module. We can similarly define $\DFmd$ as the dual of $\DFm$, and define the $\bullet$ and $\odot$-actions of $\DFm$ on $\DFmd$, which commute. Indeed, there is a canonical embedding of $\DFm\subset \mathbf{D}_{F_m}$ (one can view $\mathbf{D}_{F_m}$ as a completion of $\DFm$) induced by the embedding $R[\La]\subset R[[\La]]\cong R[[\La]]_{F_m} $ sending $1-e^{-\la}$ to $x_\la$. 

Denote 
\[Y_J^m=\left(\sum_{w\in W_J}\de_w^m\right)\prod_{\al\in\Sigma_J^+}\tfrac{1}{1-e^\al}\in \DFm, \quad \pt^m_J=Y^m_J\bullet \pt^m=Y_J^m\odot\pt^m, \]
where 
\[\pt^m=\left(\prod_{\al>0}(1-e^\al)\right)f_e\in \DFmd\cong K_T(G/B). \]
We similarly define $Y^m_{J/J'}$, and denote $A^m_{J/J'}=Y^m_{J/J'}\bullet\_: \DFmd\to \DFmd$, which gives the composition $K_T(G/B)\to K_T(G/P_J)\to K_T(G/B)$. Note that $\pt^m_J\in K_T(G/P_{J})$ is the class of the structure sheaf of  $eP_J/P_J$.  We will frequently use that $\DFmdP{J}=Y_J^m\bullet \DFmd$ and $Y_J^m\bullet\pt^m=Y_J^m\odot\pt^m$, see~\eqref{bullodot}.

\subsection{Modules over the Hecke algebra}
In $\DFm$, we define elements \[
\tau_i=(t-t^{-1})X^m_i+t^{-1}\de^m_i,\quad i\in [n].
\]
Straightforward computations show that $\tau_i$ satisfy the braid relations and the quadratic relation $\tau_i^2=(t^{-1}-t)\tau_i+1$. So the subalgebra $H$  generated by $\tau_i$  for all $i$ is isomorphic to the classical Iwahori-Hecke algebra, and the subalgebra $\mathbf{H}\subset Q_{W,m}$ generated by $R[\La]$ and $H$ is isomorphic to the affine Hecke algebra.

\begin{lem}\label{lem:YJtau} 
If $j\in J$, then $ \tau_jY^m_J=t^{- 1}Y_J^m$. 
\end{lem}
\begin{proof} By definition we know that $\de^m_jY^m_J=Y^m_J$, so $X_j^mY_J^m=0$. The conclusion then follows.
\end{proof}

Via the embedding $H \subset\DFm$ and the $\odot$-action of $\DFm$ on $\DFmdP{J}$, there is a $\odot$-action of $H$ on $\DFmdP{J}$.
Denote $\calM^{J}:=H\odot \pt^m_J$, and $\calM:=\calM^{\emptyset}$.  Consider the map of $H$-modules
\[
\varphi^{J}: H \to \calM^{J},\quad z\mapsto z\odot \pt^m_J.
\]

\begin{lem}\label{lem:DFhH}
The map $\varphi:H\to \calM,\quad  z\mapsto z\odot \pt^m$ is an isomorphism of $H$-modules.
\end{lem}

\begin{proof} 
From \cite[Theorem~10.13]{CZZ1} we know that $\DFmd$ is a free $\DFm$-module with basis $\{\pt^m\}$, so $\varphi$ is an embedding, hence the conclusion follows.
\end{proof}

\begin{lem}\label{lem:mdiag}We have $\calM^{J}\subset \DFmdP{J}$, and for any $J'\subset J$, we have the following commutative diagram, where the vertical map in the middle is an $H$-module homomorphism.
\[
\xymatrix{ H \ar@{->>}[rrr]^{\varphi^{J'}}\ar@{->>}[drrr]_{\varphi^{J}} &&&\calM^{J'}\ar@{->>}[d]^{A^m_{J/J'}}\ar@{^(->}[r] & \DFmdP{J'}\ar@{->>}[d]^{A^m_{J/J'}}\\
 &&& \calM^{J}\ar@{^(->}[r] & \DFmdP{J}}
\]
\end{lem}
\begin{proof}
Since the $\odot$-action commutes with the $\bullet$-action, we have
\[
\calM^{J}=H\odot \pt_J^m= H\odot  Y_J^m\bullet \pt^m=Y_J^m\bullet (H\odot \pt^m)\subsetneq Y_J^m\bullet \DFmd= \DFmdP{J}.
\]  
Since $A_{J/J'}^m=Y^m_{J/J'}\bullet\_$, so the commutativity between $\bullet$ and $\odot$ also implies that $A^m_{J/J'}$ is a map of $H$-modules.
\end{proof}

Let $\tau_v$ be  the basis of $H$. For any $v\in W^J$, we define
\[
\frakm^{J}_v= \varphi^{J}(\tau_v)=\tau_v\odot \pt^m_J\in \calM^{J},\quad \frakm_e^{J}=\pt^m_J, \quad \frakm_v:=\frakm^{\emptyset}_{v}\in \calM.
\]

\begin{lem}\label{lem:tauaction}
The $\odot$-action of $H$ on the basis $\frakm_v^{J}$ is given by: 
\[
\tau_j \odot\frakm^{J}_v=\left\{\begin{array}{ll} (t^{-1}-t)\frakm^{J}_{v}+\frakm^{J}_{s_jv}, & \text{if }\ell(s_jv)< \ell(v)\,,\;\;\; \text{in which case }s_jv\in W^J;\\
\frakm^{J}_{s_jv} , & \text{if } \ell(s_jv)> \ell(v) \text{ and } s_jv\in W^J;\\
t^{-1}\frakm^{J}_{v},	   & \text{if } \ell(s_jv)> \ell(v) \text{ and } s_jv=vs_i \text{ for some }i\in J.\end{array}\right.
\]
Consequently,  $\tau_{vw} \odot\frakm^{J}_{e}=t^{- \ell(w)}\frakm^J_{v}$ for $v\in W^J, w\in W_J$, and $\{\frakm_v^{J}\}_{v\in W^J}$ is a basis of $\calM^{J}$ as an $R$-module.
\end{lem}

\begin{proof}
By definition, we have 
\[
\tau_j\odot\frakm^{J}_ v=\tau_j \odot(\tau_v\odot \pt^m_J)=(\tau_j\tau_v)\odot \pt^m_J.
\]
If $\ell(s_jv)< \ell(v)$, then $s_jv\in W^J$ and $\tau_v=\tau_j\tau_{s_jv}$. Therefore, 
\[
\tau_j\tau_v=(\tau_j)^2\tau_{s_jv}=((t^{-1}-t)\tau_j+1)\tau_{s_jv}=(t^{-1}-t)\tau_v+\tau_{s_jv},
\]
which proves the first case.
	
If $\ell(s_jv)> \ell(v)$ and $s_jv\in W^J$, then $\tau_j\tau_v=\tau_{s_jv}$, hence $\tau_j \odot\frakm^{J}_ v=\frakm^{J}_{s_jv}$.
	
If $\ell(s_jv)> \ell(v)$ and $s_jv=vs_i$ with $i\in J$, then $\tau_j\tau_{v}=\tau_{v}\tau_{i}$. We have 
\begin{gather*}\tau_i\odot \pt^m_J=\tau_i\odot Y^m_J\odot \pt^m=(\tau_i Y_J^m)\odot\pt^m\overset{\sharp}=  t^{- 1} Y^m_J\odot \pt^m=  t^{- 1}\pt^m_J,\end{gather*}
where $\sharp$ follows from Lemma~\ref{lem:YJtau}, so $\tau_j\odot\frakm^{J}_ v= t^{- 1}\frakm^{J}_ v$. 

The last part follows inductively.
\end{proof}

\begin{cor}\label{cor:Am} For $J\supset J'$, we have $A^m_{J/J'}(\frakm^{J'}_v)= t^{- \ell(w)}\frakm^{J}_{u}$ if $v\in W^{J'}$ with $v=uw, u\in W^J, w\in W_J$. In particular, if $v \in W^{J}$, then 
\[
A^m_{J/J'}(\frakm_{v}^{J'})=\frakm^{J}_{v}.
\]
\end{cor}
\begin{proof} We have
\begin{align*}
A_{J/J'}^m(\frakm_v^{J'}) &=Y^m_{J/J'}\bullet \tau_v\odot Y^m_{J'}\bullet \pt^m \overset{\sharp_1} =\tau_v\odot Y^m_{J/J'}\bullet Y^m_{J'}\bullet\pt^m\\
&\overset{\sharp_2}=\tau_v\odot\pt^m_J=\tau_{uw}\odot\frakm^{J}_{e}\overset{\sharp_3}= t^{-\ell(w)}\frakm^{J }_{u},
\end{align*}
where $\sharp_1$ follows from Lemma~\ref{actcommute}, $\sharp_2$ from \eqref{prop1}, and $\sharp_3$ from  Lemma~\ref{lem:tauaction}. 
\end{proof}

\begin{rem}\label{rem:taum}Beside $H$, there is another copy of the Iwahori-Hecke algebra inside $\DFm$, denoted $H^-$. It is generated by \[\tau_i^-=(t-t^{-1})X_i^m-t\de_i^m\in \DFm.\] The elements $\frakm^{J,-}_{v}:=\tau_v^-\odot\pt_J^m$, for $v\in W^J$ form a basis of $\calM^{J,-}:=H^-\odot\pt^m_J\subset \DFmdP{J}$ as an $R$-module.  We will sometimes write $H^+$ for $H$, and correspondingly $\tau_v^+$, $\calM^{J,+}$, and $\frakm^{J,+}_{v}$ for $\tau_v$, $\calM^{J}$, and $\frakm^{J}_{v}$. 

Let $\sigma$ be the involution of $R$ defined by $t\mapsto -t^{-1}$. It induces an automorphism of $Q_m$ by $t\mapsto t^{-1}$, $x_\la\mapsto x_\la$ and, hence, an automorphism of $Q_{W,m}$ with $\sigma(\de_w^m)=\de_w^m$. By definition $\sigma(X_i^m)=X_i^m$, so $\sigma(\DFm)=\DFm$. We have $\sigma(\tau_i^\pm)=\tau_i^\mp$, so $\sigma$ is an isomorphism between $H^+$ and $H^-$. Moreover, it induces an automorphism of $(\DFm^\star)^{W_J}$ by  $z\odot \pt^m_J\mapsto \sigma(z)\odot \pt^m_J$, $z\in \DFm$. Then $\sigma(\frakm^{J,+}_v)=\frakm^{J,-}_v$, so $\sigma$ induces an isomorphism between $\calM^{J,+}$ and $\calM^{J,-}$.

Observe that Lemmas~\ref{lem:DFhH} and~\ref{lem:mdiag} hold for the ${(\cdot)}^-$ version, by means of the isomorphism $\sigma$.   The $t^{-1}$ in Lemma~\ref{lem:YJtau} should be replaced by $-t$, the $t^{-1}$ (resp. $t^{-\ell(w)}$) in the third line of Lemma~\ref{lem:tauaction} should be replaced by $-t$, (resp. $(-t)^{\ell(w)}$), and the $t^{-\ell(w)}$ in Corollary~\ref{cor:Am} should be replaced by $(-t)^{\ell(w)}$. 
\end{rem}

\begin{rem}\label{tradgen} The traditional generators of the Iwahori-Hecke algebra $H$ are $T_i^\pm:=t^{-1}\tau^\pm_i$. In other words,
\[
T^+_i=(1-t^{-2})X_i^m+t^{-2}\de_i^m=(1-q)X_i^m+q\de_i^m, \quad T_i^-=(1-q)X_i^m-\de_i^m, \quad q:=t^{-2}\,,
\]
and they satisfy $(T_i^\pm)^2=(q-1)T^\pm_i+q$.  Moreover, we have bases $\{T_w^+\:\mid\:w\in W\}$ and $\{T_w^-\:\mid\:w\in W\}$ for $H^+$ and $H^-$, respectively.
\end{rem}
 
\subsection{Deodhar's modules revisited}
We now explain the correspondence with the setup in Deodhar's work \cite{D87}. Consider a generic Iwahori-Hecke algebra with generators $T_i$. For each $J\subset \Pi$, Deodhar defined two $H$-modules $M^{J}$ using explicit $\Z[q,q^{-1}]$-basis $\{m^{J}_v\}_{v\in W^J}$, and  defined the parabolic Kazhdan-Lusztig basis. Let  $H_J\subset H$ be the Iwahori-Hecke algebra associated to the sub-system determined by $J\subset \Pi$, which acts on $\Z[q,q^{-1}]$ by the maps  $h_\pm:H_J\to \Z[q, q^{-1}]$  sending $T_i$ to $q$ and $-1$, respectively (here $h_+(T_i)=q$ and  $h_-(T_i)=-1$). Then Deodhar's modules are  $M^{J,\pm}=Ind_{H^J}^H\Z[q,q^{-1}]$, and the basis elements are $m^{J,\pm}_v=T_v\otimes 1$.

We now observe that the $\odot$-actions of $H^\pm$ on the bases $\frakm_v^{J,\pm}$ (see Lemma~\ref{lem:tauaction}) match the classical actions of the Iwahori-Hecke algebra on the parabolic modules $M^{J,\pm}$. So we have the following correspondence.

\begin{lem} With $t^{-2}=q$ and $H^\pm \cong H$, there are isomorphisms of $H$-modules
\[
M^{J, +}\stackrel\cong\to \mathcal{M}^{J,+}\quad\text{ and }\quad M^{J,-}\stackrel\cong\to \calM^{J,-} 
\]
which map $m^{J, +}_v\mapsto T_v^+\odot\pt^m_J=t^{-\ell(v)}\frakm^{J,+}_{v}$ and $m^{J,-}_v\mapsto T_v^-\odot\pt_J^m=t^{-\ell(v)}\frakm^{J,-}_{v}$ for any $v\in W^J$, respectively. 
Moreover, Deodhar's maps $\varphi_J$ are precisely our $\varphi^J$, and Deodhar's $\varphi_{J,L}$ is precisely our $A_{J/J'}^m$ where $L=J'\subset J$. 
\end{lem}

\begin{rem}
We also have
\[t\leftrightarrow v, \quad \tau_i^+\leftrightarrow H_i, \quad\calM^{J,+}\leftrightarrow \mathcal{M}, \quad \calM^{J,-}\leftrightarrow \mathcal{N}\]
where the right-hand side notations are taken from Soergel~\cite{So97}. 
\end{rem}

\subsection{The Kazhdan-Lusztig basis}  We recall several well-known facts about the {\em Kazhdan-Lusztig basis}. There is an involution on the Iwahori-Hecke algebra $H$  defined as 
\[
\overline{t}=t^{-1}, ~\overline{\tau_i}=\tau_i^{-1}, ~\overline{z_1z_2}=\overline{z_1}\cdot \overline{z_2}, ~z_1, z_2\in H.
\]
 We use $\gamma_v$ to denote the Kazhdan-Lusztig basis, that is, $\gamma_v$ is invariant under the involution and 
\[\gamma_v\in \tau_v+\sum_{w<v}t\Z[t]\tau_w.\]
In particular, we have $\gamma_i=\tau_i+t$.
We write
\[
\gamma_v=\sum_{w\le v} t^{\ell(v)-\ell(w)}P_{w,v}\tau_w, \quad P_{w,v}\in \Z[t^{-1}],
\]
where $P_{w,v}$ are the {\em Kazhdan-Lusztig polynomials}; these are traditionally written as polynomials in $q$, where $q=t^{-2}$, as mentioned above. It is known that $P_{v,w_J}=1$ for all $v\le w_J$, so
\[
\gamma_J:=\gamma_{w_J}=\sum_{w\in W_J}t^{\ell(w_J)-\ell(w)}\tau_w\,.
\]
Defining
\[
\gamma_{J/J'}:=\sum_{v\in W^{J'}\cap W_J}t^{\ell(w_{J/J'})-\ell(v)}\tau_v\in H\,,
\]
it is not difficult to see that $\gamma_{J}=\gamma_{J/J'}\gamma_{J'}$.

The following is a classical result restated and reproved in a simple way using our setup; it is needed in the proof of Theorem~\ref{thm:YPi} below.

\begin{prop} \label{prop:gammaJ} The following hold.
\begin{enumerate}
\item[{\rm (i)}] $\de_i^m(1+t^{-1}\tau_i)=1+t^{-1}\tau_i$.
\item[{\rm (ii)}] $
\gamma_{w_\circ}=t^{-\ell(w_\circ)}\left(\sum_{w\in W}\de_w^m\right)\prod_{\al>0}\tfrac{t^2-e^{\al}}{1-e^\al}\,.
$
\end{enumerate}
\end{prop}
\begin{proof}(i) It follows from the identities $(\de_i^m)^2=1$, and 
\[\de_i^m X_i^m=\tfrac{1}{x_{-\al_i}}(1-\de_i^m)=\left(\tfrac{1}{x_{\al_i}}-1\right)(\de_i^m-1)=1+X_i^m-\de_i^m\,.\]

(ii)  We can write $\gamma_{w_\circ}=\sum_{w\in W}\de_w^m a_w$ with $a_w\in Q_m$. We have $\gamma_{w_\circ}\in \tau_{w_\circ}+\sum_{v<w_\circ}t\Z[t]\tau_v$, so $a_{w_\circ}$ is the coefficient of $\de_{w_\circ}^m$ in the expansion of $\tau_{w_\circ}$ in the basis $\de_w^m$. Note that $\gamma_i=t^{-1}(\de_i^m+1)\tfrac{t^2-e^{\al_i}}{1-e^{\al_i}}$. By reasoning as in the proof of \cite[Lemma~5.4]{CZZ}, we show that 
\[a_{w_\circ}=t^{-\ell(w_\circ)}\prod_{\al>0}\tfrac{t^2-e^\al}{1-e^\al}.\]
To finish the proof, it suffices to show that $a_w=a_{w_\circ}$ for any $w$. This is equivalent to $\de_i^m\gamma_{w_\circ}=\gamma_{w_\circ}$ for any $s_i$. Denote $\null^iW=\{v\in W\,\mid\,\ell(s_iv)>\ell(v)\}$. We have
\begin{flalign*}
\de_i^m\gamma_{w_\circ}&=\de_i^m\sum_{v\in \null^iW}(t^{\ell(w_\circ)-\ell(v)}\tau_v+t^{\ell(w_\circ)-\ell(v)-1}\tau_i\tau_v)\\
&=\sum_{v\in \null{}^iW}t^{\ell(w_\circ)-\ell(v)}\de_i^m(1+t^{-1}\tau_i)\tau_v\overset{\sharp}=\sum_{v\in {}^iW}t^{\ell(w_\circ)-\ell(v)}(1+t^{-1}\tau_i)\tau_v=\gamma_{w_\circ}\,.
\end{flalign*}
Here identity $\sharp$ follows from part (i). This concludes the proof.
\end{proof}

\subsection{Other Kazhdan-Lusztig bases}\label{rem:gammam}
There are other ways of defining a Kazhdan-Lusztig basis, as follows:
\begin{gather*}
\gamma^{+,-}_v=\sum_{w\le v}(-t)^{\ell(w)-\ell(v)}P^{+,-}_{w,v}\tau^+_w\in H^+\,, \;\;~P^{+,-}_{w,v}\in \Z[t]\,,\\
\gamma^{-,-}_v=\sum_{w\le v}(-t)^{\ell(w)-\ell(v)}P^{-,-}_{w,v}\tau^-_w\in H^-\,, \;\;~P^{-,-}_{w,v}\in \Z[t]\,,\\
\gamma^{-,+}_v=\sum_{w\le v}t^{\ell(v)-\ell(w)}P^{-,+}_{w,v}\tau^-_w\in H^-\,, \;\;~P^{-,+}_{w,v}\in \Z[t^{-1}]\,;
\end{gather*}
where the above conditions are in addition to the invariance under the Kazhdan-Lusztig involution. The first sign of $\gamma$ and $P$ corresponds to the sign of $\tau$, and the second one to the choice of $t^{-1}\Z[t^{-1}]$ or $t\Z[t]$. Under this convention, the Kazhdan-Lusztig basis and polynomials defined above coincide with $\gamma^{+,+}\in H^+$ and $P^{+,+}$. 

Since the isomorphism $\sigma$ defined in Remark~\ref{rem:taum} commutes with the Kazhdan-Lusztig involution,  $\sigma(\gamma^{+,+})$ and $\sigma(\gamma^{+,-})$ are also invariant under the involution, and therefore
\[
\sigma(\gamma_v^{+,+})=\gamma_v^{-,-}\,,\;\;\;\;\sigma(\gamma_v^{+,-})=\gamma_v^{-,+}\,,\quad P^{-,-}_{w,v}=\sigma(P_{w,v}^{+,+}), \quad P_{w,v}^{-,+}=\sigma(P_{w,v}^{+,-}).
\]
Applying $\sigma$ to the identity in Proposition~\ref{prop:gammaJ}.(ii), we have
\[
\gamma^{-,-}_{w_\circ}=\sum_{w\in W}(-t)^{\ell(w)-\ell(w_{\circ})}\tau^-_w=(-t)^{\ell(w_\circ)}\left(\sum_{w\in W}\de_w^m\right)\prod_{\al>0}\tfrac{t^{-2}-e^\al}{1-e^\al}\,.
\] 
For $\gamma^{+,-}$ and $\gamma^{-,+}$, we have a similar identity to the first one, but not to the second one.

\section{The Kazhdan-Lusztig theory and hyperbolic formal group laws}

In this section, we give a functorial treatment of the parabolic Kazhdan-Lusztig basis in the new setup of the oriented cohomology of flag varieties. The main results below are new, and they are  direct consequences of this setup.

\subsection{From the multiplicative to the generic hyperbolic formal group law}\label{sec:FmtoFh}

We first relate the realizations of the Iwahori-Hecke algebra $H$ inside the Demazure algebras associated with the multiplicative formal group law $F_m$ (see \S\ref{sec:Deodmult}) and the generic hyperbolic formal group law $F_t$; this is based on two morphisms from $F_t$ to $F_m$. Then we identify and study Deodhar's modules inside the dual of the Demazure algebra for $F_t$, which is isomorphic to the corresponding cohomology  of $G/B$.

We first introduce some notation for the remainder of the paper. Let $\mu:=t+t^{-1}$,  $\Rr:=\Z[t,t^{-1}, \tfrac{1}{\mu}]$, and consider $F_m$ as a FGL over $\Rr$; also recall the notation $u:=\mu^{-2}$. Let $S_m$ and $S_t$ be the formal group algebra for $F_m$ and $F_t$, respectively. For example, $S_m=\Rr[[\La]]$ if we identify $x_\la\in S_m$ with $1-e^{-\la}\in \Rr[[\La]]$ for any $\la\in \La$. 

We use $X_i^m, Y_i^m$ and $X_i^t, Y_i^t$ to denote the corresponding divided difference and push-pull elements for $F_m$ and $F_t$, respectively.   Let $\Dh\subset \DFh$ be the subalgebra generated by $X_i^t$ for all $i$ (so $\Dh$ is the formal Demazure algebra associated to the hyperbolic formal group law). It follows from  \cite[Example~4.12]{Le15} that $\Dh$ is actually an $\Rr$-module with basis $\{X^t_{I_w}\mid w\in W\}$, a fact that  does not hold for a general FGL. 

Denote \[
\pt^t_J=Y_J^t\bullet\pt^t=Y_J^t\odot \pt^t\in \DFhdP{J}\subset \DFhd.
\]
The automorphism $\sigma$ of $\Rr$ with $\sigma(t)=-t^{-1}$, defined in Remark~\ref{rem:taum}, extends to an automorphism of $S_t$ with $\sigma(x_\la)=x_\la$, hence induces an automorphism of $\Dh$ and also of $\DFh$, satisfying that $\sigma(Y_i^t)=Y^t_i$. We also have $\sigma(Y_J)=Y_J$. Moreover, it induces an automorphism of $\DFhd\cong \DFh\bullet \pt^t$ by $\sigma(z\bullet \pt^t)=\sigma(z)\bullet \pt^t$. From the definition of $\odot$-action, and the fact that $\iota$ commutes with the involution $\sigma$,  we also have 
\[\sigma(z\odot\pt_J^t)=\sigma(Y_J\bullet\iota(z)\bullet \pt^t)=\sigma(Y_J)\iota(\sigma(z))\bullet \pt^t=\sigma(z)\odot \pt^t_J, \quad z\in \DFh. 
\] 
 
\begin{lem} There is a morphism of FGLs $g\colon F_t\to F_m$ over $\Rr$, defined by 
\[g(x)=\tfrac{(1-t^2)x}{x-(t^2+1)},\] so that $F_m(g(x),g(y))=g(F_t(x,y))$.
\end{lem}
It follows from \cite[Lemma~2.6]{CPZ} that $g$ induces a $W$-equivariant embedding of rings  
\[
\psi=\psi^+\colon S_m\hookrightarrow S_t\text{ by }f(x_\la)\mapsto f(g(x_\la))\text{ for any }\la\in \La, f(x)\in \Rr[[x]].
\] 
Note that it is not an isomorphism unless one inverts $t^2-1$ in $\Rr$. The map $\psi$ induces a morphism of algebras $\Qq_{F_m}\to \Qq_{F_t}':=\Rr[\tfrac{1}{1-t^2}][[\La]]_{F_t}$, and, hence, a morphism of twisted group algebras \[\psi\colon \Qq_{F_m}\rtimes \Rr[W]\to \Qq'_{F_t}\rtimes \Rr[W],\;\text{ where }p\in \Qq_{F_m}\mapsto\psi(p)\in \Qq'_{F_t}\text{ and }\de_i^m\mapsto \de^t_i.\] 
By definition we have 
\begin{equation}\label{psitau}\psi(\tau_i)=\mu Y^t_i-t=\mu X_i^t+t^{-1}\,.\end{equation}
Since the Hecke algebra $H$ is generated by $\tau_i$ and $\Dh$ is generated by $Y_i^t$, we obtain an isomorphism 
\[\psi: H\subset \Qq_{W,F_m} \overset{\sim}\longrightarrow \Dh\subset \Qq_{W,F_t}'.\]
Indeed, it also induces an embedding of the affine Hecke algebra $\mathbf{H}$ into  the affine hyperbolic Demazure algebra $\DFh$.

As in \eqref{prop2}, we have  $Y_i^tY^t_J=Y^t_J$ for any $i\in J$, so we have 
\begin{equation}\label{eq:3}\psi(\tau_i)Y^t_J=t^{-1}Y^t_J\,,\end{equation} which is the analogue of Lemma~\ref{lem:YJtau}.

\begin{rem}\label{rem:psim}
There is another morphism $F_t\to F_m$ defined by $g^-(x):=\tfrac{(1-t^{-2})x}{-(t^{-2}+1)+x}$. The map $\psi^-\colon\Qq_{W,F_m}\to \Qq_{W,F_t}'$ determined by $g^-$ satisfies that $\psi^-=\sigma\circ\psi^+\circ \sigma$, so it induces an isomorphism $H^-\stackrel{\sim}\to \Dh$. Indeed, it is precisely the other isomorphism mentioned in \cite[Proposition~9.2]{CZZ1}, i.e., it satisfies $\tau_i^-\mapsto -\mu Y^t_i+t^{-1}$.
In summary, we have the following commutative diagram ($\psi^+$ is the $\psi$ defined before).
\[
\xymatrix{\Dh \ar@/^2pc/[rrr]^\sigma_\sim\ar@{^(->}[r] &\DFh \ar[r] ^\sigma _\sim& \DFh & \Dh\ar@{_(->}[l]\\
   H^+\ar@/_2pc/[rrr]^\sigma_\sim \ar@{^(->}[r] \ar[u]^{\psi^+}_\sim & \DFm \ar@{^(->}[u]^{\psi^+}\ar[r]^\sigma_\sim & \DFm \ar@{^(->}[u]^{\psi^-} & H^-\ar@{_(->}[l] \ar[u]^{\psi^-}_\sim 
}
\]
\end{rem}

We define the $H$-module $\calN^J:=\psi(H)\odot \pt^t_J\subset \DFhd$ and define $\frakn^J_v:=\psi(\tau_v)\odot \pt^t_J$.  In particular, $\frakn^J_e=\pt^t_J$. 
Because of \eqref{eq:3}, all results of \S\ref{sec:Deodmult} hold if one replaces $\calM^J$ by $\calN^J$,  $\pt^m_J$ by $\pt^t_J$  and $\frakm^J_v$ by $\frakn_v^J$,  respectively.  Indeed, $\calM^J\cong \calN^J$ as $H$-modules by identifying $\frakm^J_v$ with $\frakn^J_v$. So $\{\frakn_v^J\}_{v\in W^J}$  is a basis of $\calN^J$. Furthermore, for $J'\subset J$, we have the following commutative diagram.
\[
\xymatrix{H\ar@{->>}[rrr]^{z\mapsto \psi(z)\odot\_}\ar@{->>}[drrr]_{z\mapsto \psi(z)\odot\_} &&&\calN^{J'}\ar@{->>}[d]^{A^t_{J/J'}}\ar@{^(->}[r] & \DFhdP{J'}\ar@{->>}[d]^{A^t_{J/J'}}\\
 &&& \calN^{J}\ar@{^(->}[r] & \DFhdP{J}}
\]
\begin{rem}
By using the isomorphisms $\psi: H\cong \Dh$, it is not difficult to see that for the generic hyperbolic FGL, the $\ker \partial_0$ of the complex~\eqref{eq:mainseq} is a free $\Ss$-module of rank $1$, generated by $\sum_{v\in W} t^{ -\ell(v)}\frakn_v$.
\end{rem}

\begin{thm}\label{thm:YPi} We have
\[Y^t_\Pi=\mu^{-\ell(w_\circ)}\psi(\gamma_{w_\circ}).\]
\end{thm}

\begin{proof}
Using  the identification $e^{\al}=1-x_{-\al}$ in $S_m$, straightforward computations show that \[\psi\left(t^{-1}\tfrac{t^2-e^\al}{1-e^\al}\right)=\psi\left(\tfrac{t-t^{-1}}{x_{-\al}}+t^{-1}\right)=\tfrac{\mu}{x_{-\al}}\in S_t\,.\]
Therefore, using Proposition~\ref{prop:gammaJ}, we obtain
\[\psi(\gamma_{w_\circ})=\sum_{w\in W}\de_w^t\prod_{\al>0}\psi\left(t^{-1}\tfrac{t^2-e^\al}{1-e^\al}\right)=\sum_{w\in W}\de_w^t\prod_{\al>0}\tfrac{\mu}{x_{-\al}}=\mu^{\ell(w_\circ)}Y^t_\Pi. \qedhere\]
\end{proof}

\begin{cor}\label{cor:1} Let $J'\subset J$. Then we have
\begin{enumerate} 
\item[{\rm (1)}] $\mu^{-\ell(w_J)}\psi(\gamma_{J})=Y^t_J$. In particular, $Y_J^t\in \Dh$.
\item[{\rm (2)}] $\mu^{-\ell(w_{J/J'})}\psi(\gamma_{J/J'})Y^t_{J'}=Y^t_J$.
\item[{\rm (3)}] $\calN^{J}=\psi(H\gamma_J)\odot\pt^t\subset \DFhd$. 
\item[{\rm (4)}] The embedding $i_{J/J'}\colon \DFhdP{J}\to \DFhdP{J'}$ restricts to a map $i_{J/J'}\colon \calN^{J}\subset \calN^{J'}$.
\item[{\rm (5)}] $\psi(z)\odot \pt_J^t=\mu^{-\ell(w_{J/J'})}\psi(z\gamma_{J/J'})\odot\pt^t_{J'}$. In particular, $\psi(z)\odot\pt_J^t=\mu^{-\ell(w_J)}\psi(z\gamma_J)\odot\pt^t$ and $\frakn_v^J=\mu^{-\ell(w_J)}\psi(\tau_v\gamma_J)\odot \pt^t$.
\end{enumerate}
\end{cor}
\begin{proof}(1) Follows similar as in Theorem~\ref{thm:YPi} replacing $w_\circ$ by $w_J$.  The second property follows since $\psi(H)\subset\Dh$.

(2) Follows from Theorem~\ref{thm:YPi} and the identities $Y^t	_{J/J'}Y^t_{J'}=Y^t_J$ and $\gamma_{J/J'}\gamma_{J'}=\gamma_J$.

(3) Follows from the definition of $\calN^J$.

(4) Follows from part (2) and the definition of $\calN^J$.

(5) We have
\begin{flalign*}
\psi(z) \odot \pt_J^t&\overset{\sharp_1}=\psi(z)\odot \mu^{-\ell(w_J)}\psi(\gamma_J)\odot \pt^t\\
&\overset{\sharp_2}= \mu^{-\ell(w_{J})}\psi(z\gamma_{J/J'})\odot\psi(\gamma_{J'})\odot\pt^t\\
&\overset{\sharp_3}=\mu^{-\ell(w_{J/J'})}\psi(z\gamma_{J/J'})\odot \mu^{-\ell(w_{J'})}\psi(\gamma_{J'})\odot\pt^t\\
&\overset{\sharp_4}=\mu^{-\ell(w_{J/J'})}\psi(z\gamma_{J/J'})\odot\pt^t_{J'},
\end{flalign*} 
where $\sharp_1$ and $\sharp_4$ follow from part (2), $\sharp_2$ follows from the identity $\gamma_{J}=\gamma_{J/J'}\gamma_{J'}$, and $\sharp_3$ follows from the identity $\ell(w_J)=\ell(w_{J/J'})+\ell(w_{J'})$.
\end{proof}
\begin{rem}
  Corollary~\ref{cor:1}~(1)  shows that  $\mu^{-\ell(w_{J/J'})}\psi(\gamma_{J/J'})\in \DFh$ behaves similarly to $Y_{J/J'}\in \Qq_{F_t,W}$, while $Y_{J/J'}$ may not even belong to $\DFh$ when $J\supsetneq J'\supsetneq \emptyset$. This is an important advantage. For example, it can be used to study the functoriality of the module $\calN^{J}$ as below.
\end{rem}

From Corollary~\ref{cor:1}~(4) we have the following commutative diagram.
\[
\xymatrix{\calN^J \ar[d]_{i_{J/J'}} \ar@{^(->}[r]& \DFhdP{J}\ar@{^(->}[r] \ar@{^(->}[d]_{i_{J/J'}} &\DFhd\ar@{=}[d]\\
\calN^{J'}\ar@{^(->}[r] & \DFhdP{J'}\ar@{^(->}[r] & \DFhd}
\]

\begin{rem}Note that $\calM^J$ does not satisfy the property in the previous diagram, that is, in general the embedding ${i_{J/J'}}:\DFmdP{J}\to \DFmdP{J'} $ does not map $\calM^J$  into $\calM^{J'}$.
\end{rem}

\subsection{Parabolic KL-Schubert classes}\label{subsec:parKL} We now recall Deodhar's definition of the {\em parabolic Kazhdan-Lusztig (KL) basis} and define  the parabolic KL-Schubert classes for the generic hyperbolic FGL in a canonical way, independent of choices of reduced words, unlike the parabolic Bott-Samelson classes.

 We define the Kazhdan-Lusztig involution on $\calN^{J}$
\[
\overline{\frakn_e^{J}}=\frakn_e^{J}, \quad \overline{\psi(z)\odot\frakn_e^{J}}=\psi(\overline{z})\odot\frakn_e^{J}, \quad z\in H.
\]
Then $\overline{\psi(z)\odot f}=\psi(\overline{z})\odot\overline{f}, z\in H, f\in \calN^{J}$. In particular, if $z$ is invariant, then so is $\psi(z)\odot \frakn_e^J$. Since $\gamma_J$ is invariant, so $i_{J/J'}$ commutes with the involution.

\begin{dfn}\label{def:CJ}For $v\in W^J$, let $C^{J}_v\in \frakn^{J}_v+\sum_{w<v} t\Z[t]\frakn^{J}_w\in \calN^J$ which is invariant under the involution.   By \cite[Theorem~3.1]{So97} $C^{J}_v$ are uniquely determined by these conditions. Denote $C_v=C^{\emptyset}_v$.
\end{dfn}
Write
\begin{equation}\label{defpkls}
C^{J}_v=\sum_{w\le v, w\in W^J}t^{\ell(v)-\ell(w)}P^{J}_{w,v}\frakn^{J}_w, \quad P^{J}_{w,v}\in \Z[t^{-1}]\,,
\end{equation}
where $P^{J}_{w,v}$ are the {\em parabolic Kazhdan-Lusztig polynomials}.
By the uniqueness of the Kazhdan-Lusztig basis,  we see that 
\begin{equation}\label{eq:CKL}C_v=\psi(\gamma_v)\odot\pt^t\in \calN.\end{equation}
For $J\neq \emptyset$, this is not true. However, see Corollary~\ref{cor:C2} below.

\begin{dfn}\label{defpklsc} For a  $v\in W^J$, we view $\mu^{-\ell(v)}C^{J}_v$ under the canonical embedding of $\calN^J$ into $\DFhdP{J}\simeq \hh_T(G/P_J)$ and call it
the {\em parabolic KL-Schubert class}. 
\end{dfn}

As a short summary, we have
\[
\frakn_v^J=\mu^{-\ell(w_J)}\psi(\tau_v\gamma_J)\odot\pt^t, \quad C_v^J=\mu^{-\ell(w_J)}\psi(\gamma_{vw_J})\odot\pt^t.
\]

\begin{rem} (i) By definition, the parabolic KL-Schubert classes form an $\Rr$-basis of $\calN^J$, and a $S_t$-basis of $\DFhdP{J}$. 

 (ii) In the $G/B$ case, in \cite{LZ1} we used $\psi(\gamma_{v^{-1}})\bullet{\pt^t}$ instead of \eqref{eq:CKL} to define the KL-Schubert classes. The two definitions turn out to be equivalent, as seen below:
\begin{flalign*}
C_v&=\psi(\gamma_v)\odot{\pt^t}=\left(\sum_{w\le v} t^{\ell(v)-\ell(w)}P_{w,v}\,\psi(\tau_w)\right)\odot{\pt^t}\\
&\overset{\sharp_1}=\left(\sum_{w\le v} t^{\ell(v)-\ell(w)}P_{w,v}\,\psi(\tau_{w^{-1}})\right)\bullet{\pt^t}\\
&\overset{\sharp_2}=\left(\sum_{w\le v} t^{\ell(v)-\ell(w)}P_{w^{-1},v^{-1}}\,\psi(\tau_{w^{-1}})\right)\bullet{\pt^t}=\psi(\gamma_{v^{-1}})\bullet{\pt^t}\,.
\end{flalign*}
Here $\sharp_1$ follows from \eqref{psitau} and the identity $X_i^t\bullet \pt^t=X_i^t\odot \pt^t$,  
while $\sharp_2$ is based on a classical fact which can be found, for instance, in \cite[Chapter~5,~Exercise~12]{BB}. 
\end{rem}

\begin{ex}\label{expkls} We calculate the parabolic KL-Schubert classes corresponding to the example in \S\ref{secex}.

All the relevant Kazhdan-Lusztig polynomials are equal to $1$, except for 
\[P_{{e},s_1s_3s_2}^J=P_{s_1s_3,s_1s_3s_2s_1s_3}=1+q=1+t^{-2}\,;\]
see Proposition~\ref{prop:KLpara} below. By plugging \eqref{psitau} into the definition \eqref{defpkls} of the parabolic KL-Schubert classes and cancelling terms, we obtain:
\begin{align*}&C_{e}^J={\rm pt}_J^t=\zeta_{e}^J\,,\;\;\;\mu^{-1}C_{s_2}^J=Y^t_2\odot{\rm pt}_J^t=\zeta_{s_2}^J\,,\\
&\mu^{-2}C_{s_is_2}^J=\left(Y_i^tY^t_2-\mu^{-1}tY^t_i+\mu^{-2}t^2\right)\odot{\rm pt}_J^t=\zeta_{s_is_2}^J-u\zeta_{e}^J\,,\\
&\mu^{-3}C_{s_1s_3s_2}^J=\left(Y^t_1Y^t_3Y^t_2-\mu^{-1}tY^t_1Y^t_3+\mu^{-3}(t^3+t)\right)\odot{\rm pt}_J^t=\zeta_{s_1s_3s_2}^J-u\zeta_{e}^J\,,\end{align*}
where in the third formula $i\in\{1,3\}$. 

Thus, the parabolic Schubert classes $\mu^{-2}C_{s_3s_2}^J$, $\mu^{-2}C_{s_1s_2}^J$, and $\mu^{-3}C_{s_1s_3s_2}^J$ are as follows, respectively:
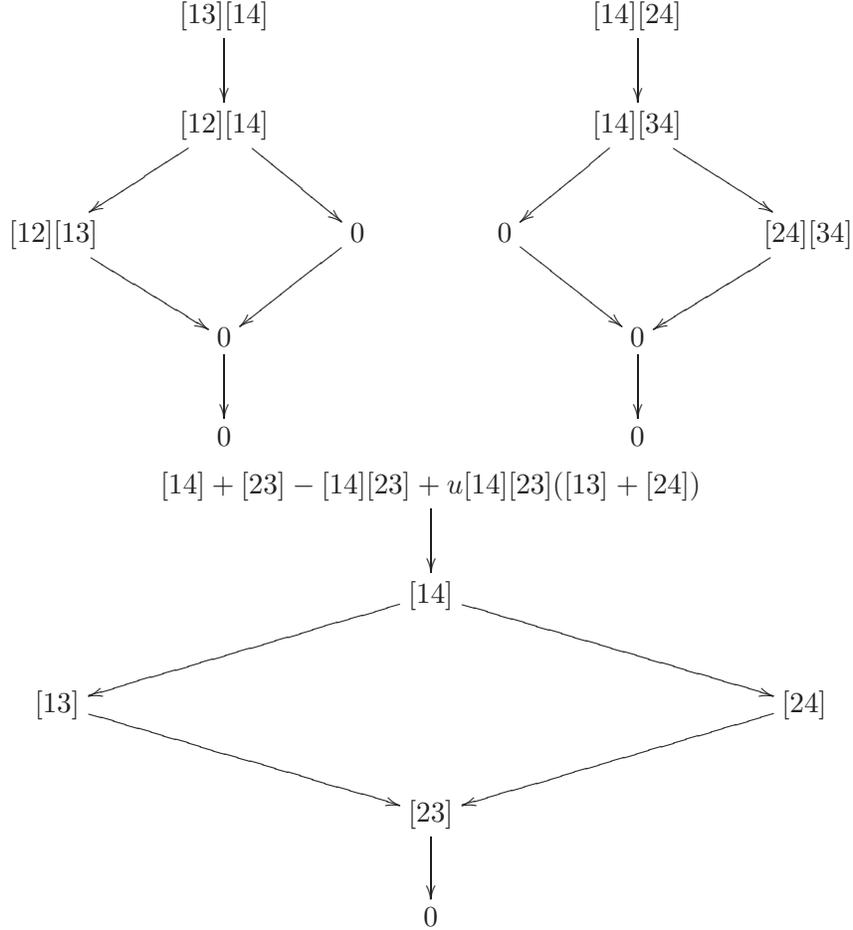
\[\scriptstyle
{\xymatrix{
&{[13][14]}\ar[d]\\
&{[12][14]}\ar[dl]\ar[dr]\\
{[12][13]}\ar[dr]&&{0}\ar[dl]\\
&{0}\ar[d]\\
&{0}
} }\qquad\qquad
{\xymatrix{
&{[14][24]}\ar[d]\\
&{[14][34]}\ar[dl]\ar[dr]\\
{0}\ar[dr]&&{[24][34]}\ar[dl]\\
&{0}\ar[d]\\
&{0}
} }
\]
\[\scriptstyle
{\xymatrix{
&{[14]+[23]-[14][23]+u[14][23]([13]+[24])}\ar[d]\\
&{[14]}\ar[dl]\ar[dr]\\
{[13]}\ar[dr]&&{[24]}\ar[dl]\\
&{[23]}\ar[d]\\
&{0}
} }
\]
Furthermore, by Corollary~\ref{cor:C2}~(2) below, we have $\mu^{-4}C_{s_2s_1s_3s_2}^J=\sum_{w\in S_4} f_w=\unit$. 
Thus, we verified Conjecture~\ref{smoothconj} below for this Grassmannian. Furthermore, note that even for the singular Schubert variety (indexed by $s_1s_3s_2$), the parabolic KL-Schubert class is given by a simpler formula than the corresponding parabolic Bott-Samelson class computed in \S\ref{secex}.
\end{ex}

\subsection{Functorial properties}\label{subsec:functor}
Let $J\supset J'$. Since $i_{J/J'}$ commutes with the Kazhdan-Lusztig involution, by the uniqueness of the parabolic Kazhdan-Lusztig basis, we recover the following classical result.

\begin{thm}\label{thm:C}For any $J\supset J'$ and $v\in W^J$, via the embedding $i_{J/J'}:\calN^J\to \calN^{J'}$, we have
\[
C^{J}_v=\mu^{-\ell(w_{J/J'})}C^{J'}_{vw_{J/J'}}\,. 
\]
\end{thm}

\begin{cor}\label{cor:C2}\hfill
\begin{enumerate} 
\item[{\rm (1)}] If $v\in W^J$, we have $C^{J}_v= \mu^{-\ell(w_{J})}C_{vw_{J}}= \mu^{-\ell(w_{J})}\psi(\gamma_{vw_{J}})\odot \pt^t\in \calN$.
\item[{\rm (2)}] As an element of $\DFhd$, the expansion of the parabolic KL-Schubert class $\mu^{-\ell(v)}C^J_v$ in the $\{f_w\}$ basis has coefficients in the subring of $S_t$ corresponding to the base ring ${\mathbb Z}[u]$ (rather than $\Rr$).
\item[{\rm (3)}] For $J\supset J'$, we have 
\begin{equation}\label{eq:2}C^{J'}_{w_{J/J'}}=\mu^{\ell(w_{J/J'})}C_{e}^J= \mu^{\ell(w_{J/J'})}\pt_J^t= \mu^{\ell(w_{J/J'})}\tfrac{x_\Pi}{x_J}\sum_{w\in W_J}f_w.
\end{equation}
In particular,  $ \mu^{-\ell(w_{\Pi/J})}C^{J}_{w_{\Pi/J}}=\sum_{w\in W}f_w=\unit\in \DFhd$.
\item[{\rm (4)}] In the limit $t\to 0$ (which implies $u\to 0$), the parabolic KL-Schubert class $\mu^{-\ell(v)}C^J_v$ becomes the Bott-Samelson class $\zeta_v^J$ in $K$-theory (which coincides with the corresponding Schubert class, defined topologically).
\end{enumerate}
\end{cor}

\begin{proof}  (1) The first identity follows from Theorem~\ref{thm:C} by letting $J'=\emptyset$, and the second one follows from \eqref{eq:CKL}.

(2) This result follows from the first part of the corollary and \cite[Proposition~3.4]{LZ1}, which it generalizes. 

(3) In \eqref{eq:2}, the first identity follows from Theorem~\ref{thm:C} by letting $v={e}$, the second one follows from the fact that $C^J_e=\frakn_e^J=\pt_J^t$, and the last one follows from \cite[Lemma~6.6]{CZZ1}. 

The last part follows from \eqref{eq:2}.

(4) This result follows from the first part of the corollary and \cite[Corollary~3.6]{LZ1}, which it generalizes. Recalling the notation in \S\ref{heckeactions}, we also use the fact that $\zeta_v^J=A_J(\zeta_v)$ coincides with $\zeta_{vw_J}$ in $K$-theory, for $v\in W^J$. For example, see Remark~\ref{rempbs}~(ii).  
\end{proof}

\subsection{Smoothness}\label{subsec:smooth}
In the $G/B$ case, it was conjectured in \cite{LZ1} that if the Schubert variety corresponding to $v$ is smooth, then $C_v$ is equal to the corresponding topologically defined Schubert class. We now generalize this conjecture to the parabolic case, and then prove some special cases; other results are also derived along the way.

We first recall from \cite{CP94,CK03} the localization formula for the class $[X(v)]$ of a smooth Schubert variety in $G/P$, where $v,w\in W^J$:
\begin{equation}\label{smooth}
[X(v)]_w=\tfrac{\displaystyle{\prod_{\al\in w\left(\Sigma_{\Pi/J}^-\right)}x_{\al}}}{\displaystyle{\prod_{\stackrel{\al\in w\left(\Sigma_{\Pi/J}^-\right)}{s_\al w\le vw_J}} x_{\al} }}\,,  
\end{equation}
if $w\le v$, and otherwise $[X(v)]_w=0$.

\begin{conj}\label{smoothconj} If the Schubert variety $X(v)$ in $G/P$ is smooth, then the topologically defined Schubert class $[X(v)]$ given by {\rm \eqref{smooth}} coincides with the parabolic KL-Schubert class $\mu^{-\ell(v)}C^{J}_v$. 
\end{conj}

Recall that the Weyl group for the root system of type $C_n$ is the group of signed permutations, represented (in the window notation) as words of length $n$ with letters $\{1,\ldots,n,\overline{n},\ldots,\overline{1}\}$; we use the indexing of the corresponding Dynkin diagram $\Pi:=\{0,\ldots,n-1\}$, where the $0$-node is the sign change in position $1$. 

\begin{cor}\label{cor:smooth} Conjecture~\ref{smoothconj} is true in the following cases (which all correspond to non-singular Schubert varieties): 
\begin{enumerate}
\item[{\rm (1)}] in all types for Schubert varieties indexed by $w_{J/J'}$ in $G/P_{J'}$;
\item[{\rm (2)}] in type $A_{n-1}$ in the maximal parabolic case $J=\Pi\setminus\{n-k\}$ (with $1<k\le n-1$), for $v\in W^J$ of the form $[k,k+1,\ldots,n-1,1,2,\ldots,k-1,n]$ (in one-line notation);
\item[{\rm (3)}] in type $C_n$ in the maximal parabolic case $J=\Pi\setminus\{k\}$ (with $1\le k\le n-1$), for $v\in W^J$ of the form $[1,2,\ldots,k-1,n,\overline{n-1},\overline{n-2},\ldots,\overline{k}]$ (in the window notation);
\item[{\rm (4)}] for the complex projective spaces.
\end{enumerate}
In addition, the conjecture is true for a Schubert variety in $G/P_J$ indexed by $v\in W^J$ if and only if it is true for the one indexed by $vw_{J/J'}$ in $G/P_{J'}$, for $J'\subset J$.
\end{cor}

\begin{proof} We need \cite[Lemma~4.1~(2)]{LZ1}; this says that, given a non-singular Schubert variety, if the coefficients in the expansion of a KL-Schubert class (as an element of $\DFhd$) are products $\prod_\al x_{\al}$ over some subsets of the negative roots, then the class coincides with the topologically defined Schubert class. This lemma was proved in the $G/B$ case, but it extends in a straightforward way to the parabolic case, using Corollary~\ref{cor:C2}~(4). This lemma combined with Theorem~\ref{thm:C} imply the last statement of the corollary. Moreover, based on the lemma, part (1) follows from Corollary~\ref{cor:C2}~(3), while parts (2) and (3) follow from \cite[Theorem~3.14~(2)]{LZ1} and Corollary~\ref{cor:C2}~(1). Finally, part (4) is the content of Theorem~\ref{thmcpn}.
\end{proof}

\subsection{Positivity}\label{subsec:positive}
A positivity property in the $G/B$ case was conjectured \cite[Conjecture~6.4]{LZ} for the (hyperbolic) Bott-Samelson classes, and in \cite[Conjecture~3.9]{LZ1} for the corresponding KL-Schubert classes. Here we conjecture the same property for the parabolic KL-Schubert classes. If the conjecture is true, it would be interesting to find the geometric reason behind it.

\begin{conj}
The coefficient of $f_w$ in the expansion of the parabolic KL-Schubert class $\mu^{-\ell(v)}C^J_v$ (as an element of $\DFhd$, where $w\le v$) can be expressed as a (possibly infinite) sum of monomials in $x_{\al}$, where $\al$ are negative roots, such that the coefficient of each monomial is of the form 
\[(-1)^{k-(N-\ell(v))}\, c\, u^{(m-k)/2}\,;\]
 here $c$ is a positive integer, $m$ is the degree of the monomial, $N-\ell(v)\le k\le m$, $m-k$ is even, and $N$ is the cardinality of $\Sigma_{\Pi/J}^+$.  
\end{conj}

\begin{rem} (i) The above positivity property is a generalization of the one in $K$-theory which is made explicit in Graham's formula \cite{graekt} for the localization of Schubert classes at torus fixed points, cf. also~\cite{LZ}.

(ii) The conjecture does not hold for the parabolic Bott-Samelson classes, cf. the example in \S\ref{secex}. However, when passing to the corresponding KL-Schubert classes, the terms violating the positivity condition disappear, so the conjecture holds; see Example~\ref{expkls}.
\end{rem}

\subsection{Other results in Kazhdan-Lusztig theory reinterpreted} \label{sec:paraBorel}
The following corollary was known to experts, since it is equivalent to Proposition~\ref{prop:KLpara}, which is a classical result in \cite{D87}. Here we give a new proof  using the embedding $i_J$. 
\begin{cor}
\begin{enumerate}
\item[{\rm (i)}] For any $z\in H$, we have 
\begin{gather*}
\psi(z)\odot\frakn_{e}^{J}=C^{J}_v\Longleftrightarrow z\gamma_{J}=\gamma_{vw_J}.
\end{gather*}
\item[{\rm (ii)}] For any $v\in W^J$, there exists $z\in \sum_{w\in W^J}\Rr \tau_w $ such that $\gamma_{vw_J}=z\gamma_{w_J}$.
\end{enumerate}
\end{cor}
\begin{proof}(i) Inside $\calN$, by Corollary~\ref{cor:1}~(5), we know that $z\odot\frakn_e^{J}= \mu^{-\ell(w_J)}\psi(z\gamma_J)\odot \pt^t,$ and by Theorem~\ref{thm:C} we have $C^{J}_v= \mu^{-\ell(w_J)}\psi(\gamma_{vw_J})\odot \pt^t$.
By Lemma~\ref{lem:DFhH} the map $H\to \calN, z\mapsto \psi(z)\odot \pt^t$ is an isomorphism, 
we get that  $\psi(z)\odot\frakn_e^{J}=C^{J}_v$ if and only if  $z\gamma_J=\gamma_{vw_J}$. 

(ii) The set $\{\frakn^{J}_w\}_{w\in W^J}$ is a $\Rr$-basis of $\calN^{J}$, and we know that $\frakn^{J}_w=\psi(\tau_w)\odot \frakn_e^{J}$, so the map $\sum_{w\in W^J}\Rr\tau_w\to \calN^J, z\mapsto \psi(z)\odot \frakn_e^J$ is surjective. Therefore, for any $v\in W^J$, there exists $z\in \sum_{w\in W^J}\Rr \tau_w$ such that $\psi(z)\odot \frakn_e^{J}=C^{J}_v$. Then part (i) implies that $z\gamma_{w_J}=\gamma_{vw_J}$.
\end{proof}

The following result appears as \cite[Proposition~3.4]{D87} and \cite[Proposition~3.4]{So97}; it follows from Theorem~\ref{thm:C} and the uniqueness of the parabolic Kazhdan-Lusztig basis.
 
\begin{prop}\label{prop:KLpara} \cite[Proposition 3.4]{D87}
	$P_{wu, vw_J}=P^{J}_{w,v}$ for $w,v\in W^J, u\in W_J$.
	\end{prop}

\subsection{Other parabolic KL-Schubert bases} 
Continuing Remark~\ref{rem:taum}, \S\ref{rem:gammam}, and Remark~\ref{rem:psim}, one can define $\calN^{J,-}=\psi^-(H^-)\odot\pt_J^t$ and $\frakn^{J,-}_v=\psi^-(\tau_v^-)\odot\pt^t_J, v\in W^J$, then the $\calN^J, \frakn^{J}_v$ defined in \S\ref{sec:FmtoFh} can be denoted by $\calN^{J,+}$ and $\frakn_v^{J,+}$, respectively. One can similarly define $C_v^{J,+,-}\in \calN^{J,+}$ and $C_v^{J,-,-}, C_v^{J,-,+}$ in $\calN^{J,-}$, and then Definition~\ref{def:CJ}  above can be denoted by $C_v^{J,+,+}\in \calN^{J,+}$. All properties concerning $\gamma^{+,+}$ and $C^{J,+,+}$ hold for the $(-,-)$ version, i.e., after replacing  $\psi^+$ by $\psi^-$, $\gamma^{+, +}$ by $\gamma^{-,-}$, $\calN^{J,+}$ by $\calN^{J,-}$, $\frakn^{J,+}$ by $\frakn^{J,-}$, $C^{J,+,+}$ by $C^{J,-,-}$, and  $\mu$ by $-\mu$. For example, $C_v^{-,-}=\psi^-(\gamma^{-,-}_v)\odot\pt^t$, and  the analogue to Theorem~\ref{thm:YPi} is
\[
Y_\Pi=(-\mu)^{-\ell(w_\circ)}\psi^-(\gamma^{-,-}_{w_\circ}).
\]
Considering the involution $\sigma$ on $\DFhd$ defined at the beginning of \S\ref{sec:FmtoFh}, we have 
\[
\sigma(\frakn^{J,+}_v)=\frakn^{J,-}_v, \quad \sigma(C^{J,+,+}_{v})=C^{J,-,-}_{v}, \quad \sigma(C^{J,+,-}_v)=C^{J,-,+}_v. 
\]
By applying $\sigma$ to the identity in Theorem~\ref{thm:C}, we obtain the following version:
\begin{equation}\label{eq:paramm} C^{J,-,-}_v=\sigma(C^{J,+,+}_v)=(-\mu)^{-\ell(w_{J/J'})}C^{J',-,-}_{vw_{J/J'}}\,.
\end{equation}

In fact, we can get the following more precise relationship between the above bases:
\[C_v^{J,-,-}=\sigma(C_v^{J,+,+})=\sigma(\mu^{\ell(v)})\,\sigma(\mu^{-\ell(v)}C_{v}^{J,+,+})\overset{\sharp}=(-1)^{\ell(v)}C_v^{J,+,+}\,.\]
Here $\sharp$ is based on Corollary~\ref{cor:C2}~(2), which implies that $\mu^{-\ell(v)}C_{v}^{J,+,+}$ is fixed by $\sigma$. Hence, by considering $(-\mu)^{-\ell(v)}C_v^{J,-,-}$, we obtain the same class as in Definition~\ref{defpklsc}. 

It is not difficult to check that $A_J:\calN^{\pm}\to \calN^{J,\pm}$ commutes with the Kazhdan-Lusztig involution, 
so we have 
\[
A_J(C_v^{+,-})=C_v^{J,+,-}, \quad A_J(C^{-,+}_v)=C_v^{J,-,+},\quad v\in W^J,
\]
which imply
\begin{equation}\label{eq:PJpm}
P^{J,+,-}_{u,v}=\sum_{w\in W_J}(-1)^{\ell(w)}P^{+,-}_{uw, v}, \quad P^{J,-,+}_{u,v}=\sum_{w\in W_J}(-1)^{\ell(w)}P^{-,+}_{uw, v}.
\end{equation}
Classically, in \cite{D87}, with $q=t^{-2}$, Deodhar considered $\gamma^{+,-}\in H^+, \gamma^{-,-}\in H^-$ and $C^{J,+,-}\in \calN^{J,+}, C^{J,-,-}\in \calN^{J,-}$. For example, in {\it loc.cit.}, Proposition~3.4 coincides with our Proposition~\ref{prop:KLpara} for $P^{-,-}$ which can be derived from \eqref{eq:paramm},  and Remark~3.8 is the first identity of \eqref{eq:PJpm}. In \cite{So97}, Proposition~3.4.1 and 3.4.2 correspond to our Proposition~3.4 and the second identity of \eqref{eq:PJpm}.

\begin{rem}
It seems that in order to generalize the Kazhdan-Lusztig polynomials to the parabolic case, $H^+$ (resp. $H^-$) and $P^{J,+,+}$ (resp. $P^{J,-,-}$) are the correct objects to generalize $P^+:=P^{+,+}$ (resp. $P^-:=P^{-,-}$). Moreover, we believe that the parabolic KL-Schubert class $\mu^{-\ell(v)}C_v^{J,+,+}$ will have geometric significance, cf. Conjecture~\ref{smoothconj}, Corollary~\ref{cor:C2}, and Corollary~\ref{cor:smooth}. 
\end{rem}

\subsection{Other hyperbolic formal group laws} So far we focused entirely on the hyperbolic formal group law \eqref{hypfgl} with the special choice of the parameters $\mu_1,\mu_2$ in \eqref{thypfgl}. So it is natural to ask what happens in general; for instance, the case of the Lorentz formal group law, given by $\mu_1=0$, is also an important one. 

We will show that the general case is related to the {\em generic Iwahori-Hecke algebra}, which depends on two parameters $q_1,q_2$. This is defined in the same way as the usual one, except that we use the more general relation $(T_i-q_1)(T_i-q_2)=0$. The classical case is recovered by setting $q_1=q$ and $q_2=-1$. It is well-known that one can define a Kazhdan-Lusztig basis for the generic Iwahori-Hecke algebra (via reduction to the classical case) whenever $-q_1q_2$ is a square in the base ring \cite{KL79}.

We have seen that the hyperbolic formal group law corresponds to the 2-parameter Todd genus $T_{\al,\be}$, where $\mu_1=\al+\be$ and $\mu_2=-\al\be$. So we assume that we have such elements $\al,\be$ in the base ring. Imposing the type $A$ braid relations on elements $T_i=aY_i+b$ in the hyperbolic Demazure algebra amounts to a quadratic equation for $b$. The solutions are $b=-a\al$ and $b=-a\be$. Correspondingly, using $Y_i^2=\mu_1 Y_i$, we easily obtain
\[(T_i+\al)(T_i-\be)=0\,,\;\;\;\;\mbox{and}\;\;\;\;(T_i-\al)(T_i+\be)=0\,.\]
Thus, the Kazhdan-Lusztig basis can be defined as long as $-\mu_2$ is a square in the base ring. In this case, all of the above constructions apply with minor changes. 

Recall that our previous choice \eqref{thypfgl} for $\mu_1,\mu_2$ corresponds to $\al=\tfrac{t}{t+t^{-1}}$ and $\be=\tfrac{t^{-1}}{t+t^{-1}}$. In this case, the two choices of $T_i$ above are precisely the $T_i^+$ and $T_i^-$, cf. Remark~\ref{tradgen}.

\subsection{Complex projective spaces}\label{subsec:projective}

We compute the Bott-Samelson classes and KL-Schubert classes corresponding to the complex projective space ${\mathbb P}^{n-1}$, so $W=S_n$, $J=\{2,\ldots,n-1\}$, and 
\[W^J=\{e,\,s_1,\,s_2s_1,\,\ldots,s_{n-1}\ldots s_1\}\,.\] 
For simplicity, we let $w_i:=s_i\ldots s_1$, where $w_0:=e$. We use the same notation as in \S\ref{secex} and Example~\ref{expkls}. Note that $\pt_J^t=[12]\ldots[1n]\sum_{w\in W_J} f_w$. 

We start with the Bott-Samelson classes $\zeta_{w_i}^J=Y^t_i\ldots Y^t_1\odot\pt_J^t$. We let 
\[\zeta_{w_i}^J=\sum_{j=1}^{i+1} q_{ij} \sum_{w\in W_J}f_{w_{j-1}w}\,.\]
Throughout, we make the convention that sums are $0$ and products are $1$ if the initial and final values of the corresponding index are in contradiction with the variation of this index (for instance, the initial value is strictly greater than the final value when the index is increasing).  

\begin{prop}\label{bscp} We have
\[q_{ij}=[j,i+2]\ldots[jn]+u\sum_{k=j+1}^{i}[jk]\ldots[jn]\,.\]
\end{prop}
\begin{proof} We proceed by induction on $i$, which starts at $i=0$. Assume that the formula holds for $i-1\ge 0$ and all $j=1,\ldots,i$. We compute $\zeta_{w_i}^J=Y^t_i\odot\zeta_{w_{i-1}}^J$. We have the following two cases.

{\em Case~{\rm 1}}: $j=i$ or $j=i+1$. We have
\begin{align*}q_{ii}&=\tfrac{[i,i+1][i,i+2]\ldots[in]}{[i,i+1]}=[i,i+2]\ldots[in]\,,\\q_{i,i+1}&=\tfrac{[i+1,i][i+1,i+2]\ldots[i+1,n]}{[i+1,i]}=[i+1,i+2]\ldots[i+1,n]\,.\end{align*}

{\em Case~{\rm 2}}: $j<i$. We have
\begin{align*}
q_{ij}&=\tfrac{1}{[i,i+1]}\left([j,i+1][j,i+2]\ldots[jn]+u\sum_{k=j+1}^{i-1}[jk]\ldots[jn]\right)\\
&+\tfrac{1}{[i+1,i]}\left([ji][j,i+2]\ldots[jn]+u\sum_{k=j+1}^{i-1}[jk]\ldots[jn]\right)\\
&=[j,i+2]\ldots[jn]+u[ji][j,i+1][j,i+2]\ldots[jn]+u\sum_{k=j+1}^{i-1}[jk]\ldots[jn]\,,
\end{align*}
as needed; here we used the first identity in \eqref{ident} and $k_\al=1$.
\end{proof}

For the calculation of the KL-Schubert classes, we need several lemmas.

\begin{lem}\label{lemma1} We have
\[\tau_k\ldots\tau_1\odot \pt_J^t=\left(\mu^k\,Y^t_k\ldots Y^t_1-\sum_{l=0}^{k-1} t^{2-k+l}\mu^l\,Y^t_l\ldots Y^t_1\right)\odot\pt_J^t\,.\]
\end{lem}
\begin{proof}
Since $k_\al=1$, we note first that $Y^t_l\odot\pt_J^t=\pt_J^t$ for $l>1$. This implies that, for $l>j+1$, we have
\[\tau_lY^t_j\ldots Y^t_1\odot\pt_J^t=(\mu Y^t_l-t)Y^t_j\ldots Y^t_1\odot\pt_J^t=t^{-1}Y^t_j\ldots Y_1^t\odot\pt_J^t\,.\]
The result follows easily by using this fact in order to expand
\[\tau_k\ldots\tau_1\odot \pt_J^t=(\mu Y^t_k-t)\ldots(\mu Y^t_1-t)\odot\pt_J^t\,.\]
\end{proof}

Let 
\[\mu_k:=\tfrac{t^{k+1}-t^{-(k+1)}}{t-t^{-1}}=t^k+t^{k-2}+\ldots+t^{-(k-2)}+t^{-k}\,.\]
This notation is local to this section, so it does not interfere with similar notation used earlier. In particular, $\mu_1=\mu$ and $\mu_0=1$. 

The following lemma expresses a KL-Schubert class in terms of Bott-Samelson classes.

\begin{lem}\label{lemma2} We have 
\begin{equation}\label{cwij}C_{w_i}^J= \mu^i\,\zeta_{w_i}^J-\sum_{k=0}^{i-2} \mu_{i-2-k}\,\mu^k\, \zeta_{w_k}^J\,.\end{equation}
\end{lem}
\begin{proof}
Since the complex projective spaces are smooth, all the corresponding parabolic Kazhdan-Lusztig polynomials are equal to $1$. So we have
\[C_{w_i}^J=\left(\sum_{k=0}^{i} t^{i-k}\tau_k\ldots\tau_1\right)\odot\pt_J^t\,.\]
By plugging the formula in Lemma~\ref{lemma1} into the previous one, we obtain:
\[C_{w_i}^J=\left(\sum_{k=0}^{i} t^{i-k} \left( \mu^k\,Y^t_k\ldots Y^t_1-\sum_{l=0}^{k-1} t^{2-k+l}\mu^l\,Y^t_l\ldots Y^t_1 \right)\right)\odot\pt_J^t\,.\]
For each $k$ between $0$ and $i-1$, we now cancel the term $t^{i-k}\mu^k\,Y^t_k\ldots Y^t_1$ in the outer sum with the term corresponding to $l=k$ in the inner sum, for the index $k+1$ in the outer sum; indeed, the latter is
\[t^{i-(k+1)}t^{2-(k+1)+k}\mu^k\,Y^t_k\ldots Y^t_1\,.\]
The only positive term surviving is $\mu^i\,Y^t_i\ldots Y^t_1$ (throughout, we have to take into account the $\odot$-action on $\pt_J^t$). By collecting the remaining terms, we obtain the desired formula. 
\end{proof}

\begin{lem}\label{lemma3} We have
\[\mu^k-\sum_{l=0}^{k-2} \mu_{k-2-l}\,\mu^l-\mu_k=0\,.\]
\end{lem}
\begin{proof}
We use induction on $k$ which starts at $k=1$ (in which case the summation in the formula is void). Assuming that the above formula holds for $k-1>0$, and multiplying it through by $\mu$, we obtain:
\[\mu^k-\sum_{l=1}^{k-2} \mu_{k-2-l}\,\mu^l=\mu_{k-1}\,\mu\,.\]
Plugging the right-hand side into the formula for $k$ to be proved, it remains to show that
\[\mu_{k-1}\,\mu=\mu_{k-2}+\mu_k\,,\;\;\;\;\;\mbox{i.e.,}\;\; (t^k-t^{-k})(t+t^{-1})=(t^{k-1}-t^{-(k-1)})+(t^{k+1}-t^{-(k+1)})\,.\]
But this is immediate.
\end{proof}

It now remains to plug the formula for Bott-Samelson classes in Proposition~\ref{bscp} into the formula for a KL-Schubert class in Lemma~\ref{lemma2}. We will show that $\mu^{-i}C_{w_i}^J$ is obtained by simply setting $u=0$ in the formula for $\zeta_{w_i}^J$.

\begin{thm}\label{thmcpn} We have
\[\mu^{-i}C_{w_i}^J=\sum_{j=1}^{i+1} [j,i+2]\ldots[jn] \sum_{w\in W_J}f_{w_{j-1}w}\,.\]
This expression agrees with the localization formula {\rm \eqref{smooth}} for the corresponding Schubert class.
\end{thm}
\begin{proof}
Let us focus on the coefficient of $f_{w_{j-1}}$ in $C_{w_i}^J$, where $1\le j\le i+1$ is fixed. For simplicity, we write $[[k]]:=[jk]\ldots[jn]$. We  calculate the mentioned coefficient as mentioned above. We write \eqref{cwij} by replacing the summation index $k$ by $p$, and note that $\zeta_{w_p}^J$ contains $f_{w_{j-1}}$ precisely when $j-1\le p\le i-2$. Thus, the cases $j=i$ and $j=i+1$ are immediate, so we assume $j<i$. 

The desired coefficient is expressed as follows:
\begin{align}
&\mu^i[[i+2]]+u\mu^i\sum_{k=j+1}^i[[k]]-\sum_{p=j-1}^{i-2}\mu_{i-2-p}\,\mu^p\left([[p+2]]+u\sum_{l=j+1}^p[[l]]\right)\nonumber\\
=&\mu^i[[i+2]]+u\mu^i\sum_{k=j+1}^i[[k]]-\sum_{p=j+1}^{i}\mu_{i-p}\,\mu^{p-2}\left([[p]]+u\sum_{l=j+1}^{p-2}[[l]]\right)\,.\label{longex}
\end{align}
The only products appearing in \eqref{longex} beside the first term are $[[k]]$ for $j+1\le k\le i$. Fixing such a value $k$, it suffices to show that the coefficient of $[[k]]$ is $0$. Let us first note that $[[k]]$ appears in the inner sum in \eqref{longex} for the index $p$ of the outer sum precisely when $j+1\le k\le p-2$, so $p$ needs to be in the range $k+2\le p\le i$. It follows that the coefficient of $[[k]]$ in \eqref{longex} is given by
\begin{align*}u\mu^i-\mu_{i-k}\,\mu^{k-2}-u\sum_{p=k+2}^i\mu_{i-p}\,\mu^{p-2}&=\mu^{i-2}-\mu_{i-k}\,\mu^ {k-2}-\sum_{p=k}^{i-2}\mu_{i-p-2}\,\mu^{p-2}\,\\&=\mu^{k-2}\left( \mu^{i-k}-\mu_{i-k}-\sum_{p=k}^{i-2}\mu_{i-p-2}\,\mu^{p-k} \right)\,.\end{align*}
The latter expression is $0$ by Lemma~\ref{lemma3}. 

We conclude the proof by easily checking that the obtained formula agrees with the localization formula~\eqref{smooth}. 
\end{proof}

\bibliographystyle{plain}

\end{document}